\documentclass[preprint,12pt]{elsarticle}

\usepackage{amssymb}
\usepackage{amsthm}
\usepackage{amsmath,latexsym}
\usepackage{setspace}
\usepackage{mathbbol}

\def\s#1{\mbox{\boldmath $#1$}}
\newtheorem{theorem}{Theorem}

\newtheorem{lemma}[theorem]{Lemma}

\newtheorem{corollary}[theorem]{Corollary}

\newtheorem{definition}[theorem]{Definition}
\newtheorem{claim}[theorem]{Claim}
\def\qed{\hfill $\Box$\\ }

\def\s#1{\mbox{\boldmath $#1$}}
\def\ds{double square}
\def\bds{balanced double square}
\def\fds{factorizable double square}
\def\fsds{FS-double square}
\def\FrSi{Fraenkel and Simpson}
\def\CrRy{Crochemore and Rytter}
\def\sbig#1{{\scalebox{0.9}{$\mathbb{s}$}}(#1)}
\def\send#1{{\scalebox{0.9}{$\mathbb{e}$}}(#1)}

\def\+{{+}}
\def\-{{-}}

\def\o#1{\overline{#1}}
\def\c#1{\widetilde #1}
\def\p#1{{\scalebox{0.7}{$\,{\cal #1}(1)$}}}
\def\q#1{{\scalebox{0.7}{$\,{\cal #1}(2)$}}}
\def\lcp#1{$lcp({#1}_1$,$\c{#1}_1)$}
\def\lcs#1{$lcs({#1}_1$,$\c{#1}_1)$}
\def\inv#1{\o{#1}_2{#1}_2{#1}_2\o{#1}_2}
\def\invf{inversion factor}
\def\amate{$\alpha$-mate}
\def\bmate{$\beta$-mate}
\def\gmate{$\gamma$-mate}
\def\dmate{$\delta$-mate}
\def\emate{$\varepsilon$-mate}
\def\eemate{super-$\varepsilon$-mate}
\def\np#1{\!\raisebox{1.5pt}{$^#1$}}
\def\aseg{$\alpha$-segment}
\def\bseg{$\beta$-segment}
\def\gseg{$\gamma$-segment}
\def\afam{$\alpha$-family}
\def\bfam{$(\alpha\+\beta)$-family}
\def\gfam{$(\alpha\+\beta\+\gamma)$-family}

\newenvironment{my_itemize}{
\begin{itemize}
  \setlength{\itemsep}{1pt}
  \setlength{\parskip}{0pt}
  \setlength{\parsep}{0pt}}{\end{itemize}
}

\newenvironment{my_enumerate}{
\begin{enumerate}
  \setlength{\itemsep}{1pt}
  \setlength{\parskip}{0pt}
  \setlength{\parsep}{0pt}}{\end{enumerate}
}

\journal{Journal of Discrete Applied Mathematics}

\begin{document}

\begin{frontmatter}
\title{How many double squares can a string contain?}
\author{Antoine Deza}
\ead{deza@mcmaster.ca}

\author{Frantisek Franek}
\ead{franek@mcmaster.ca}

\author{Adrien Thierry}
\ead{thierraa@mcmaster.ca}

\address{
Advanced Optimization Laboratory \\
Department of Computing and Software\\
McMaster University, Hamilton, Ontario, Canada}

\begin{abstract}
\noindent
Counting the types of squares rather than their occurrences, we consider the problem of bounding the number of distinct squares in a string.
\FrSi\ showed in 1998  that a string of length $n$ contains at most $2n$ distinct squares. Ilie presented in 2007  an asymptotic 
upper bound of  $2n\-\Theta(\log\ n)$. We show that a string of length $n$ contains at most $\lfloor11n/6\rfloor$ distinct squares. 
This new upper bound is obtained by investigating the combinatorial 
structure of \ds{s} and showing 
that a string of length $n$  contains at most $\lfloor5n/6\rfloor$  
particular double squares. In addition, 
the established structural properties provide a novel  proof of \FrSi's result.
\end{abstract}


\end{frontmatter}

\section{Introduction}
A square in a string is a tandem repetition of the form $u^2 = uu$.
The repeating part, $u$, is referred to as the \emph{generator} of
the square $u^2$. If the generator $u$ is \emph{primitive}, i.e. not
a repetition of a string, then the square is called \emph{primitively rooted}.
The problem of counting the types of squares in a string of length $n$ -- later referred to as \emph{the number of distinct squares problem} --
was introduced by \FrSi~\cite{FS98} in 1998 who showed that the number of distinct squares in a string of length $n$ is at most $2n$. 
Their proof relies on a lemma by \CrRy~\cite{CR95}  describing the relationship among the sizes of three primitively rooted squares starting
at the same position. Not using \CrRy's Lemma, Ilie~\cite{I05}  provided an alternative proof  of \FrSi's result
before presenting in~\cite{I07} an asymptotic upper bound of $2n\-\Theta(\log\ n)$ for sufficiently large $n$.
A $d$-step approach to this problem introducing the size $d$ of the alphabet as a parameter in
addition to the length $n$ of the string was proposed in~\cite{DFxx}. 
Considering the maximum number $\sigma_d(n) $ of distinct primitively rooted squares over all strings of length $n$ with exactly $d$ distinct symbols, it is conjectured there that $\sigma_d(n) \leq n\-d$.  
Note that the number of non-primitively rooted squares, i.e. squares whose generators are repetitions, is bounded by  $\lfloor n/2 \rfloor\-1$, see Kubica et al.~\cite{KRRW12}. 

A configuration of two squares $u^2$ and $U^2$ starting at the same position and so that $|u|<|U|<2|u|<2|U|$ has been investigated in different contexts. For instance,  the configuration
of such two squares with a third one is investigated in~\cite{FFSS12,KS12} with the intention of  providing a position where a third square could not start in order to tackle the maximum number of runs conjecture. 
Within the computational framework introduced in~\cite{PSC2012-11}, such configurations are investigated in~\cite{L13} to enhance the determination of $\sigma_d(n)$. Such configurations of two squares are unique in the context of rightmost occurrences of squares since at most two such squares can start at the same position as shown by  \FrSi. In~\cite{L} Lam investigates 
what he calss \ds{s}, i.e. configurations of two rigthmost occurences of squares starting at the same position, in order to bound their number and thus bound the number of distinct squares.

We  present structural properties of \ds{s} arising in various contexts and coinciding with Lam's \ds{s} in the context
of rightmost occurrences which we refer to as \fsds{s}. The structural properties of \ds{s}
presented in this paper not only give a novel proof of \FrSi's result, they allow bounding the number of \fsds{s} in a string of length $n$ by $\lfloor 5n/6\rfloor$, which in turn leads
to a new upper bound for the number of distinct squares
of $\lfloor11n/6\rfloor$.

\section{Combinatorics of double squares}
\subsection{Preliminaries}

\medskip
\noindent
We deal with finite strings over finite alphabets and index strings starting from 1. Thus $x[1]$ refers to the first symbol of a string $x$, $x[2]$ to the second etc.
We use $..$ as a range symbol, thus $x=x[1..n]$ is a string of length $n$, and $x[i..j]$ refers to the substring, also often called  \emph{factor}, starting at position $i$ and ending at position $j$. 
For a substring $y=x[i..j]$, $\sbig{y}$ respective $\send{y}$ denotes its starting, respective ending, position, i.e. $(\sbig{y},\send{y})=(i,j)$.
A substring $y=x[i..j]$ of  $x=x[1..n]$ is called a \emph{prefix} respective \emph{suffix} of $x$ if $i=1$ respective $j=n$, and is \emph{proper} if $y\neq x$, while we call it \emph{trivial} if $y$ is empty.
For a string $x$, a \emph{non-trivial power} of $x$ is a string $x^m$ for some integer $m\geq 2$, where $x^m$ represents a concatenation of $m$ copies of
$x$. In particular, $x^2$ is called a \emph{square}, and $x^3$ a \emph{cube}. 

\begin{definition}
A string $x$ is \emph{primitive} if  $x$ cannot be expressed as a non-trivial power of any string.
For any string $x$, there is a primitive string $y$ so that $x=y^m$ for some integer $m\geq 1$. Such $y$ and $m$ are unique and $y$ is called the \emph{primitive root} of $x$.
Two strings $x$ and $y$ are \emph{conjugates} if there are strings $u$ and $v$ so that $x=uv$ and $y=vu$. Note that $x$ is a trivial conjugate of
itself. Often the term \emph{rotation} is used for conjugates.
\end{definition}

Lemmas~\ref{syncpr} and~\ref{comfactor}  are folklore and presented without
proofs. 

\begin{lemma}[Synchronization principle lemma]\label{syncpr}
Given a primitive string $x$, a proper suffix $y$ of $x$, a proper prefix $z$ of $x$, and $m\geq 0$, 
there are exactly $m$ occurrences of $x$ in $yx{^m}z$.
\end{lemma}

Note that Lemma~\ref{syncpr} implies that a primitive string does not equal to any of its
conjugates.

\begin{lemma}[Common factor lemma]\label{comfactor}
For any primitive strings $x$ and $y$, if a non-trivial power of $x$ and
a non-trivial power of $y$ have a common factor of length $|x|\+|y|$,
then $x$ and $y$ are conjugates.
\end{lemma}

\subsection{Double squares}

\medskip
\begin{definition}
A configuration of two squares $u^2$ and $U^2$ in a string $x$
starting at the same position is referred to as a \emph{\ds}. In 
case that $|u|<|U|$, we say that  \emph{$(u,U)$ is a \ds}, 
i.e. the smaller generator is listed first.

 For a \ds\ $(u,U)$ in a string $x$, if $|u|<|U|<2|u|$,
we say that the squares $u^2$ and $U^2$ are \emph{proportional} 
and we call such a \ds\ \emph{balanced}.

For a \ds\ $(u,U)$, if moreover $u^2$ and $U^2$ are 
rightmost occurrences in $x$, we refer to the \ds\ $(u,U)$ as 
\emph{\fsds\ of $x$}.
\end{definition}

\medskip
Note that if $(u,U)$ is a \ds, respective \bds, in $x$ and $x$ is a substring of $y$,
then $(u,U)$ is a \ds, respective \bds, in $y$ as well.
For \fsds, due to $u^2$ being a rightmost occurrence
in $x$,  $|U|<2|u|$, as otherwise in $x$ would be a farther copy of $u^2$, 
and so every \fsds\ is automatically balanced. If $x$ is a substring of $y$,
$(u,U)$ need not be a \fsds\ in $y$; on the other hand if $x$ is a suffix
of $y$, then $(u,U)$ is a \fsds\ in $y$ as well.
We refer to the \bds{s} of rightmost occurrences  as  
\fsds{s} in recognition of \FrSi's pioneering efforts
in the problem. 

\medskip
In Lemma~\ref{ds} we shall show that certain types of \bds{s}
have a unique factorization consisting of  a nearly periodical
repetition of a primitive string. The following Lemma~\ref{uu-unique-factorization} is used in Lemma~\ref{ds} to prove uniqueness of this factorization.

\begin{lemma}
\label{uu-unique-factorization}
Let ${u_1}^pu_2={v_1}^qv_2$ where $u_1, v_1$ are primitive, $u_2$ is a non-trivial proper prefix of $u_1$, and $v_2$ is a non-trivial proper prefix of $v_1$. If $p\geq 2$ and $q\geq 2$, then $u_1=v_1$, $u_2=v_2$, and $p=q$.
\end{lemma}

\begin{proof}
Since $p\geq 2$ and $q \geq 2$, and ${u_1}^p$ and ${v_1}^q$ have a common factor of size  $|u_1|\+|v_1|$, then by Lemma~\ref{comfactor}, $u_1=v_1$. Thus, $u_2=v_2$ and $p=q$.
\end{proof}

Note that in Lemma~\ref{uu-unique-factorization}, $p\geq 2$ and $q\geq 2$ are essential conditions.
For instance, $u_1=aabb$, $u_2=aa$, and $p=2$ gives ${u_1}^pu_2=aabbaabbaa$,
and $v_1=aabbaabba$, $v_2=a$, and $q=1$ gives ${v_1}^qv_2=aabbaabbaa$; that is,  ${u_1}^pu_2={v_1}^qv_2$.

As we often need to refer to the various occurrences of the
same factor,  we use a special subscript $[1]$, $[2]$, etc to
distinguish them. For instance, $u_{[1]}$ may refer to the first occurrence
of $u$ in $u^3$, while $u_{[2]}$ would refer to the second occurrence, etc.

\medskip
Lemma~\ref{ds}  gives various contexts in which a \bds\ has
a unique factorization. 
While a weaker form of Lemma~\ref{ds}  is proven in~\cite{L13},  and item $(c)$ and
the fact the $U^2$ must be primitively rooted are proven in~\cite{L}, the uniqueness is not addressed in either. 

\begin{lemma}
\label{ds}
Let $(u,U)$ be a balanced double square. If one of the following conditions is
satisfied

\begin{my_itemize}
\leftskip=-5pt
\item[$(\text{a})$] $u$ is primitive
\item[$(\text{b})$] $U$ is primitive
\item[$(\text{c})$] $u^2$ has no further occurrence in $U^2$
\end{my_itemize}

\vspace{-8pt}
\noindent
then there is a unique primitive string $u_1$, a  unique non-trivial proper prefix $u_2$ of $u_1$, and unique integers $e_1$ and $e_2$ satisfying
$1\leq e_2\leq e_1$ such that $u = {u_1}^{e_1}u_2$ and  $U = {u_1}^{e_1}u_2{u_1}^{e_2}$. Moreover, $U$ is primitive.
\end{lemma}

\begin{proof}
Let $v_1$ denote the overlap of $U_{[1]}$ with $u_{[2]}$; that is, $u=v_1\o{v}_1$ for some $\o{v}_1$ and $U=uv_1$, see the
diagram below.

\includegraphics[scale=1]{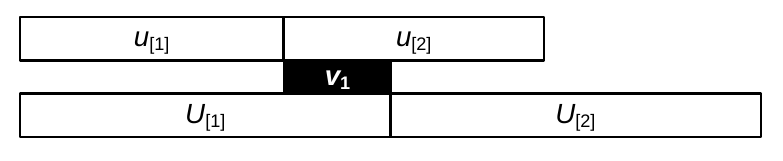}

\noindent
Thus, $u$ is a prefix of $v_1U$ and
$u={v_1}^{k}v_2$ for some prefix $v_2$ of $v_1$ and $k\geq 1$. 
Let $u_1$ be the primitive root of $v_1$. Then $v_1={u_1}^{e_2}$
for some $e_2\geq 1$. Therefore $u={u_1}^{e_1}u_2$ for
some $e_1\geq k e_2$ and some prefix $u_2$ of $u_1$.
The prefix $u_2$ must be non-trivial, as otherwise:\\
$(a)$ Let us assume that $u_2$ is the empty string. If $e_1\geq 2$, then $u=u_1^{e_1}$ and hence not primitive,
a contradiction. If $e_1=1$, then $e_2=1$ and so $U={u_1}^2$ and $u=u_1$ and
so $|U|=2|u|$, a contradiction.\\ 
$(b)$ $U={u_1}^{e_1\+e_2}$ and $e_1\+e_2\geq 2$, hence $U$ would not be primitive.\\
$(c)$  there would be a farther occurrence of $u^2={u_1}^{2e_1}$ in $U^2={u_1}^{2e_1\+2e_2}$.

\medskip
To prove the uniqueness, consider some primitive $w_1$, its non-trivial proper prefix $w_2$, and integers $f_1\geq f_2\geq 1$ such that
$u={w_1}^{f_1}w_2$ and $U={w_1}^{f_1}w_2{w_1}^{f_2}$.
If $e_1\geq 2$ and $f_1\geq 2$, then by Lemma~\ref{uu-unique-factorization}, 
$u_1=w_1$ and $e_1=f_1$ and it follows that $u_2=w_2$ and $e_2=f_2$.
If $e_1=f_1=1$, it follows that $u=u_1u_2=w_1w_2$.
Since $U=uu_1=uw_1$, $u_1=w_1$ and so $u_2=w_2$.
The remaining case corresponds to exactly one of the exponents $e_1$ and
$f_1$ being equal to 1.
Without loss of generality, we can assume that $e_1=1$ and $f_1>1$.
We have $u=u_1u_2={w_1}^{f_1}w_2$ and $U=u_1u_2u_1={w_1}^{f_1}w_2{w_1}^{f_2}$. Thus, $u_1={w_1}^{f_2}$.
As $u_1$ is primitive, $f_2=1$, and so $u_1=w_1$.
Therefore, $u_1u_2={w_1}^{f_1}w_2={u_1}^{f_1}w_2$ and so $f_1=1$, contradicting $f_1>1$.

\medskip
Let us assume that $U$ is not primitive and derive a contradiction. Thus, $U=v^n$ for some primitive $v$ and some $n\geq 2$.
It follows that $|v|\leq \frac{|U|}{2} = \frac{|{u_1}^{e_1}|\+|u_2|\+|{u_1}^{e_2}|}{2}\leq
 \frac{|{u_1}^{e_1}|\+|u_2|\+|{u_1}^{e_1}|\+|u_2|}{2}=|{u_1}^{e_1}|\+|u_2|$.
Now consider $U^2=v^{2n}={u_1}^{e_1}u_2{u_1}^{e_1\+e_2}u_2{u_1}^{e_2}$.  It follows
that ${u_1}^{e_1\+e_2}u_2$ is a factor of $v^{2n}$, $2n\geq 2$
of size $\geq |v|\+|u_1|$, $e_1\+e_2\geq 2$,
and so by Lemma~\ref{comfactor}, $u_1$ and $v$ are
conjugates, hence $u_1=v$. Thus $U=v^n={u_1}^n={u_1}^{e_1}u_2{u_1}^{e_1}$ and so
$n|u_1|=(e_1\+e_2)|u_1|\+|u_2|$, which is
impossible as $0<|u_2|<|u_1|$. Therefore, $U$ must be primitive.
\end{proof}

\begin{definition} [Notation and terminology] 
\label{notat}
If a \bds\ satisfies one of three conditions (a), (b), or (c) of Lemma~\ref{ds},
we will refer to such \ds\ as \emph{factorizable}.
We use the following notational convention for \fds{s}:
a \ds\ $\cal U$ consists of two squares $u^2$ and $U^2$, where $|u|<|U|$ and so  we refer  to $u^2$ respective $U^2$ as the \emph{shorter} respective \emph{longer}, \emph{square} of $\cal U$,  and to the starting position of $u^2$ and $U^2$ as the \emph{starting position} of $\cal U$.
The unique exponents are denoted as $\p{U}$ and $\q{U}$, the repeating primitive part of $u$ is denoted as $u_1$, the prefix of $u_1$ completing $u$ is denoted as $u_2$.  Thus $u={u_1}^{\p{U}}u_2$ and
$U=u{u_1}^{\q{U}}={u_1}^{\p{U}}u_2{u_1}^{\q{U}}$.
Since $u_2$ is a non-trivial proper prefix of $u_1$, there is complement $\o{u}_2$ of $u_2$ in $u_1$ so that $u_1=u_2\o{u}_2$.
The conjugate $\o{u}_2u_2$ of $u_1$ is denoted as $\c{u}_1$, i.e. 
$\c{u}_1=\o{u}_2u_2$.
\end{definition}

For instance, a \fds\ $\cal V$ consists of the shorter square 
$v^2$ and the
longer square $V^2$, and $v={v_1}^\p{V}v_2$ and $V={v_1}^\p{V}v_2{v_1}^\q{V}$.
We would like to point out that for any 
\fds\ $\cal U$, $|U^2|=2((\p{U}\+\q{U})|u_1|\+|u_2|)\geq
2((1+1)2+1)=10$ since $\p{U}\geq \q{U}\geq 1$, $|u_1|\geq 2$,
and $|u_2|\geq 1$. Thus, only strings of length at least 10 may contain a \fds. Note also, that by (c) of Lemma~\ref{ds}, every
\fsds\ is a \fds.
Lemma~\ref{canon1}  further specifies the structure
of a \fds, i.e. the fact that the shorter and the longer squares must have essentially different structures.

\begin{lemma}
\label{canon1}
If ${\cal U}$  is a \fds\ so that $u={v_1}^iv_2$ for some primitive
$v_1$, some  non-trivial proper prefix $v_2$ of $v_1$, and some integer $i\geq 1$;
then $U\neq{v_1}^jv_2$ for any $j\geq 1$.
\end{lemma}

\begin{proof}
Clearly, $U\ne{v_1}^jv_2$ for $j\leq i$ since  $|U|>|u|$. Thus, consider $j>i$ and
assume by contradiction that  $U={v_1}^jv_2$.
Then, for $j=i\+1$,
$U=u{u_1}^{\q{U}}={v_1}^iv_2{u_1}^{\q{U}}={v_1}^{i\+1}v_2$
and so $v_2{u_1}^{\q{U}}={v_1}v_2$. Denote by $\o{v}_2$ the complement
of $v_2$ in $v_1$, i.e. $v_1=v_2\o{v}_2$. Then
$v_2{u_1}^{\q{U}}=v_2\o{v}_2v_2$, and so
${u_1}^{\q{U}}=\o{v}_2v_2$. Since $\o{v}_2v_2$ is a conjugate of
$v_1$ and hence primitive, it follows that $\q{U}=1$ and thus $u_1=\o{v}_2v_2$.
Thus $U={v_1}^{i\+1}v_2={v_2}({\o{v}_2v_2})^{i\+1}={v_2}{u_1}^{i\+1}$ and
also $U={u_1}^{\p{U}}u_2{u_1}$, so ${u_1}^{\p{U}}u_2{u_1}={v_2}{u_1}^{i\+1}$
contradicting Lemma~\ref{syncpr} as $|v_2|<|v_1|=|u_1|$.
For $j>i\+1$, ${v_1}^iv_2v_1$ must be a prefix of ${v_1}^j$ contradicting
Lemma~\ref{syncpr}
\end{proof}

\vspace{-10pt}
Lemma~\ref{nonprimitiveds} discusses the case when the shorter square of a \fds\ is not primitively 
rooted. It shows that the size of $U$ is highly constraint.

\begin{lemma}
\label{nonprimitiveds}
Let $\cal U$ be a \fds\ so that $u=v^k$, for some primitive $v$ and some $k\geq 2$. Then $\p{U}=\q{U}=1$ and $U=v^{2k\-1}v_1$ for some non-trivial proper prefix $v_1$ of $v$. Moreover, $u_1=v^{k\-1}v_1$ and $v_1u_2=v$.
\end{lemma}

\begin{proof}
Let us assume that $\p{U}\geq 2$ and derive a contradiction.
Then $u={u_1}^{\p{U}}u_2=v^k$, giving $|u_1|<|v|$.
It follows that ${u_1}^{\p{U}}u_2$ and $v^k$ have a common factor
of length $\geq |u_1|\+|v|$ and by Lemma~\ref{comfactor}, $u_1$ and $v$ are conjugates, and
so $u_1=v$. But then $|u|=\p{U}|u_1|\+|u_2|=k|u_1|$, which is impossible as
$0<|u_2|<|u_1|$. Therefore, $\p{U}=1$ and so $\q{U}=1$.

Since $U$ is a prefix of $v^{2k}$, $U=v^tv_1$ where $k\leq t\leq 2k\-1$ and $v_1$ is a proper prefix of $v$. 
Since $U$ must be primitive by Lemma~\ref{ds}, $v_1$ must be a
non-trivial proper prefix.
If $t=2k\-1$, then we are done and the proof is complete. Let us thus
assume that $t<2k\-1$.
Then $2k\-t\geq 2$ and so the suffix $v^{2k\-t}$ of $u^2$ starts
at the same position $p$ as the suffix $v_1U=v_1v^tv_1$ of $U^2$. Therefore
factors $v^2$ (a subfactor of $v^{2k\-t}$)  and $v_1v$ (a subfactor 
of $v_1v^tv_1$) start at the same position $p$,
contradicting Lemma~\ref{syncpr} as $v$ is primitive.

Since $U=uu_1$, $U=v^{2k\-1}v_1=v^kv^{k\-1}v_1=uv^{k\-1}v_1$,
and so $u_1=v^{k\-1}v_1$. Since $u=u_1u_2$,
$v^k=v^{k\-1}v_1u_2$ and so $v_1u_2=v$.
\end{proof}

\begin{definition}
A factor $u=x[i..j]$ of $x$ can be \emph{cyclically shifted
right by 1 position} if $x[i]=x[j\+1]$. The factor $u$ can be \emph{cyclically 
shifted right by $k$ positions} if $u$ can be cyclically shifted right by 1 position and the factor $x[i\+1..j\+1]$ can be cyclically
shifted right be $k\-1$ positions. Similarly for \emph{left cyclic
shifts}. By a \emph{trivial} cyclic shift we mean a shift by $0$ positions.
\end{definition}

\medskip
Note that if $v$ is a right cyclic shift of $u$, then $u$ and $v$ are
conjugates. Similarly for left cyclic shift.

\medskip
Let $x$ contain a \fds\ $\cal U$ and let $x=y_1U^2y_2$.
To cyclically shift $\cal U$ to the right means that both $u^2$ and $U^2$ must be cyclically shifted to the right.  The maximal right cyclic shift of
$u^2$ is determined by \lcp{u}, while the maximal right cyclic shift of $U^2$ is determined 
by the $lcp(U^2,y_2)$, where $lcp(x,y)$ is the length of the \emph{largest common prefix} of $x$ and $y$.  
Similarly, to cyclically shift $\cal U$ to the left means that both $u^2$ and $U^2$ must be cyclically shifted to the left.  The  maximal left cyclic shift of $u^2$ is determined by 
\lcs{u}, while the maximal left cyclic  shift of 
$U^2$ is determined by the $lcs(U^2,y_1)$, where $lcs(x,y)$ is the length of the 
\emph{largest common suffix} of $x$ and $y$. Thus, 
\lcs{u} represents the maximal potential left cyclic shift of $u^2$, while \lcp{u} 
represents the maximal potential right cyclic shift of $u^2$. 

\begin{lemma}
\label{rot2}
For any \fds\ $\cal U$, \lcp{u}\+\lcs{u} $\leq |u_1|\-2$.
\end{lemma}

\begin{proof}
If \lcp{u}\+\lcs{u} $\geq |u_1|$, then $u_1=\c{u}_1$ contradicting the 
primitiveness of $u_1$. So 
\lcp{u}\+\lcs{u} $< |u_1|$.
Assume then that\linebreak
 \lcp{u}\+\lcs{u} $=|u_1|\-1$. Let $i=$ \lcp{u} and let $a$ be the 
symbol at position $i$ of $u_1$, i.e. $u_1[i]=a$.
Then $u_1[1.. i\-1]=\c{u}_1[1.. i\-1]$ as 
$|\{1,..i\-1\}|=$ \lcp{u}, and 
$u_1[i\+1.. |u_1|\-1]=\c{u}_1[i\+1.. |u_1|\-1]$ 
as $|\{i\+1,..,|u_1|\-1\}|=$ \lcs{u}. 
Thus, $u_1$ and $\c{u}_1$ coincide in all positions except possibly $i$.
Therefore $u_1[1..i\-1][i\+1..|u_1|\-1]$ and $\c{u}_1[1..i\-1][i\+1..|u_1|\-1]$ must have the same number of $a$'s. Since $u_1$ and $\c{u}_1$
are conjugates, they both have to have the same number of $a$'s.
Therefore $\c{u}_1[i]=a$ yielding $u_1=\c{u}_1$, and thus 
contradicting the primitiveness of $u_1$.
\end{proof}

\subsection{Inversion factors}

\medskip
\noindent
A key combinatorial property of \fds{s} is the highly 
constrained occurrences of so-called \emph{\invf s}. 
The notion of \invf\ is motivated by the two
occurrences of the  factor $\inv{u}$ in a \ds\ $\cal U$. Even though for
the purpose of this paper it would be sufficient to define \invf\ as any
cyclic shift of $\inv{u}$ which would greatly simplify the proof of 
the correspondingly simplified Lemma~\ref{invfactor},
we decided to include a more general definition of \invf\ and thus a more 
general version of Lemma~\ref{invfactor}.

\begin{definition}
Given a {\fds} ${\cal U}$, a factor of $U^2$ of length 
$2|u_1|$ starting at position $i$ is called \emph{\invf} if

$\begin{cases} U^2[i\+j]=U^2[i\+j\+|u_1|\+|u_2|] &\mbox{ for } 0\leq j< |\o{u}_2|, and \\
U^2[i\+j]=U^2[i\+j\+|u_2|] & \mbox{ for } |\o{u}_2|\leq j< |u_2|\+|\o{u}_2|. \end{cases}$\\
Note that an \invf\ of $\cal U$ has a form $\inv{v}$ 
where $|v_2|=|u_2|$ and  $|\o{v}_2|=|\o{u}_2|$.
\end{definition}

\noindent
In a \fds\ $\cal U$, \invf{s} $\inv{u}$ 
occur at positions $N_1({\cal U})$ and $N_2({\cal U})$
where

\smallskip
$N_1({\cal U})=\send{{u_1}^{\p{U}\-1}u_2}\+1=
(\p{U}\-1)|u_1|\+|u_2|\+1$ 

$N_2({\cal U})=\send{{u_1}^{\p{U}}u_2{u_1}^{\q{U}\+\p{U}\-1}u_2}\+1=
(2\p{U}\+\q{U}\-1)|u_1|\+2|u_2|\+1$.

\medskip
\noindent 
Such \invf{s} are referred to as \emph{natural}.

\medskip
Cyclic shifts of the \invf\ $\inv{u}$ are governed by the values of \lcp{u} and \lcs{u}. A cyclic shift
of an \invf\ is again an \invf. Thus, at every position of the union of the intervals  $[L_1({\cal U}),R_1({\cal U})]$ and $[L_2({\cal U}),R_2({\cal U})]$ 
there is an \invf\ of $\cal U$ starting there, where

\medskip
$L_1({\cal U})=max\ \{\ 1, N_1({\cal U})\-$\lcs{u}\ $\}$

$R_1({\cal U})=  N_1({\cal U})\+$\lcp{u}

$L_2({\cal U}) = N_2({\cal U})\-$\lcs{u}

$R_2({\cal U}) = min\ \{\ \send{U^2}\-2|u_1|\+1,N_2({\cal U})\+$\lcp{u}\ $\}$.

\medskip
\noindent 
If it is clear from the context, 
we  omit the $\cal U$ designation from 
$N_1({\cal U})$, $N_2({\cal U})$, $L_1({\cal U})$, $R_1({\cal U})$, $L_2({\cal U})$,  and $R_2({\cal U})$.
Note that $L_2\-L_1=R_2\-R_1=|U|$ and,  by Lemma~\ref{rot2}, $R_1\-L_1=R_2\-L_2\leq |u_1|\-2$. 
In addition, $L_1\leq R_1<\send{u_{[1]}}<\sbig{u_{[2]}}<\send{U^2}$ and
$\send{u_{[1]}}<\sbig{u_{[2]}}<L_2\leq R_2\leq \send{U^2}\-2|u_1|<\send{U^2}$.
A key fact is that besides the intervals 
$\s{\big[}L_1,R_1\s{\big]}$ and $\s{\big[}L_2,R_2\s{\big]}$, there are no further occurrences 
of an \invf\ in a \fds\ $\cal U$. In other words, all {\invf}s are cyclic shifts of the natural ones.

\begin{figure}
\leftskip=-20pt
\includegraphics[scale=0.8]{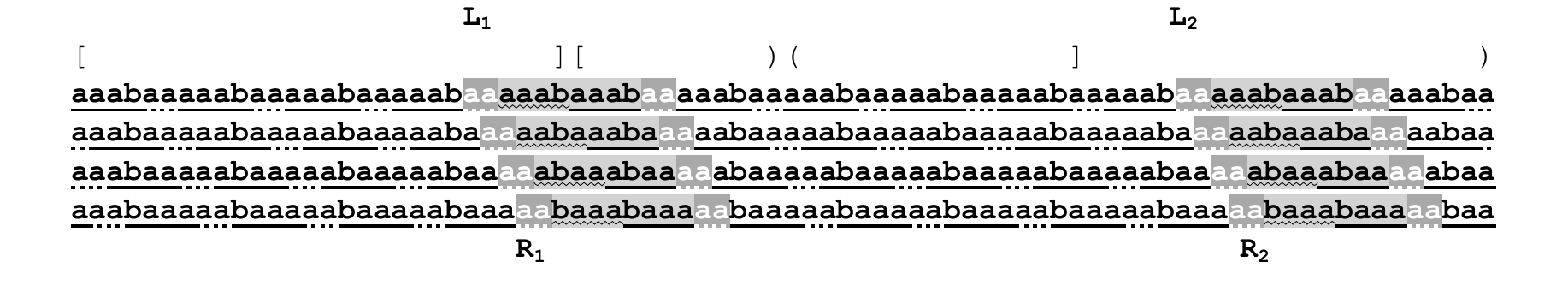}
\caption{Cyclic shifts of the {\invf} and its environment}
\end{figure}

\leftskip=0pt

See Figure~1 for an illustration where
$u_2=aaab$, $\o{u}_2=aa$,
$\p{U}=4$, and $\q{U}=2$. Consequently, $u_1=aaabaa$ and $\c{u}_1=aaaaab$, and so \lcp{u} = 3 and
\lcs{u} = 0. Thus, the \invf\
$\inv{u}=aaaaabaaabaa$ has three non-trivial right cyclic shifts
and no non-trivial left cyclic shift. Note that there are no 
other \invf{}s besides those highlighted. The configuration of brackets \verb+[  ][  ]+ indicates the shorter square while the configuration
\verb+[  )(  )+ indicates the longer square.
Also note that the environments
of the {\invf}s are shifted along: the \invf\  $\inv{v}$ is always preceded by $v_2$ (solid underline) alternating with $\o{v}_2$ (dotted underline). The leftmost piece of the environment, i.e. starting at the beginning of the string,  might just be a suffix of $v_2$ or $\o{v}_2$. Similarly,
the \invf\ $\o{v}_2v_2v_2\o{v}_2$ is always followed by $v_2$ alternating with $\o{v}_2$. The rightmost piece of
the environment, ending at the end  of the string $U^2$,  might just be a prefix of $v_2$ or $\o{v}_2$. 

\begin{lemma}[Inversion factor Lemma]\label{invfactor}
An \invf\ of a \fds\ $\cal U$ within the string $U^2$ starts at a position $i$ if 
and only if  $i\in \s{\big[}L_1({\cal U}),R_1({\cal U})\s{\big]} \cup \s{\big[}L_2({\cal U}),R_2({\cal U})\s{\big]}$.
\end{lemma}

\noindent
The rather technical proof of Lemma~\ref{invfactor} is given in Section~\ref{proof-invfactor}.

\section{Inversion factors and the problem of distinct squares}
\label{applications}

When computing the number of distinct squares, one must consider just one representative occurrence from all occurrences of each square. \FrSi~\cite{FS98} consider only the last, i.e., the rightmost 
occurrence. We consider the same context and thus will be
investigating \fsds{s}. Let us recall that \fsds{s} are factorizable which  follows from Lemma~\ref{ds}~(c).
\FrSi's theorem states that at most two rightmost occurring squares can start at the same position using Lemma~\ref{CR+FS}:

\begin{lemma}[\CrRy~\cite{CR95}, \FrSi~\cite{FS98}]
\label{CR+FS}
Let $u^2$, $v^2$, and $w^2$ be squares in a string $x$ starting at the same position such that $|u|<|v|<|w|$ and with $u$ primitive, 
then $|w|\geq |u|\+|v|$.
\end{lemma}

\noindent
Though one could prove Lemma~\ref{CR+FS} using the inversion factor Lemma~\ref{invfactor}, we follow Ilie~\cite{I05} and prove Theorem~\ref{2n} directly.

\begin{theorem}[\FrSi~\cite{FS98}, Ilie~\cite{I05}]
\label{2n}
At most two rightmost squares can start at the same position.
\end{theorem}

\begin{proof}
Let us assume by contradiction that three rightmost  squares
start at the same position: $u^2$, $U^2$, and $v^2$ such that
$|u|<|U|<|v|$. By item $(c)$ of Lemma~\ref{ds}, $u^2$ and $U^2$ form a \fds\ $\cal U$
and so $u={u_1}^{\p{U}}u_2$ and $U={u_1}^{\p{U}}u_2{u_1}^{\q{U}}$.
Since $v_{[1]}$ contains an inversion factor, $v_{[2]}$ must also 
contain an inversion factor. If the \invf\ in $v_{[2]}$ were from $[L_2,R_2]$, then
$|v|=|U|$, a contradiction.
Hence $v_{[2]}$ must not contain an inversion factor from $[L_2,R_2]$ and so 
${u_1}^{\p{U}}u_2{u_1}^{\p{U}\+\q{U}\-1}u_2$ must be a prefix of $v$. Therefore
$v_{[2]}$ contains another copy of  ${u_1}^{\p{U}}u_2{u_1}^{\p{U}}u_2=u^2$,  contradicting
the assumption that $u^2$ is a rightmost square.
\end{proof}

\vspace{-10pt}
We often need to investigate the mutual configuration of the shorter squares of two \fds{s}.
\begin{definition}
For two substrings $u$ and $v$ of a string $x$ such that\linebreak
$\sbig{u}<\sbig{v}$,
the \emph{gap} $G(u,v)$ is defined as $\sbig{v}\-\sbig{u}$ and the \emph{tail} $T(u,v)$
is defined as $\send{v}\-\send{u}$.
For two \fds{s} $\cal U$ and $\cal V$ such that\linebreak
$\sbig{{\cal U}}<\sbig{{\cal V}}$, the \emph{gap} $G({\cal U},{\cal V})=G(u,v)$
and the \emph{tail} $T({\cal U},{\cal V})=T(u,v)$.
\end{definition}

Note that $T(u,v)$ could be negative when $\send{v}<\send{u}$. If $T(u,v)\geq 0$,
then $G(u,v)v=uT(u,v)$. If it is clear from the context, we will drop the reference 
to $u$ and $v$ or $\cal U$ and $\cal V$ and use just $G$ and $T$.
Lemma~\ref{v-cases} investigates configurations consisting of an \fsds\ and a single rightmost square. In essence it says that if we have
an \fsds\, then the types and starting positions for a possible rightmost square $v^2$ are highly constraint.  
Lemma~\ref{v-cases} is needed for Lemma~\ref{vv-cases} discussing configurations of two {\fsds}s.

\onehalfspacing
\begin{lemma}
\label{v-cases}
Let $x$ be a string starting with an \fsds\ $\cal U$. Let $v^2$ be a rightmost occurrence in $x$. Then

\vspace{-10pt}
\begin{my_itemize}
\leftskip=-7pt
\item[$(a)$] If $\sbig{v_{[1]}}<R_1({\cal U})$, then there are the
following possibilities for $v^2$:

\vspace{-5pt}
\begin{my_itemize}
\leftskip=-20pt
\item[$(a_1)$] $|v| < |u|:$  in which case $v={{\widehat u}_1}\np{j}{\widehat u}_2$ for
some $1\leq j< \p{U}$ where ${{\widehat u}_2}$ is a non-trivial
proper prefix of 
${{\widehat u}_1}$ and ${\widehat{u}_1}$ respective $\widehat{u}_2$ is a cyclic shift of $u_1$ respective $u_2$ by the same number of positions in the  same direction;
\item[$(a_2)$] $|v| = |u|:$  in which case $v={{\widehat u}_1}^\p{U}{\widehat u}_2$
where ${{\widehat u}_2}$ is a non-trivial proper prefix of ${{\widehat u}_1}$ and ${\widehat{u}_1}$ respective $\widehat{u}_2$  is a cyclic shift of $u_1$ respective $u_2$ by the same number of positions in the same direction;
\item[$(a_3)$] $|u| < |v| < |U|:$ is impossible;
\item[$(a_4)$]  $|v|=|U|:$ in which case $T(u,v)\geq 0$;
\item[$(a_5)$] $|v|>|U|:$ in which case $T(u,v)\geq 0$ and either $s_1\o{u}_2u_2{u_1}^{(\p{U}\+\q{U}\-1)}{u_2}$ is a prefix of $v$ for some
suffix $s_1$ of $u_2$, or $s_1{u_1}^iu_2{u_1}^{(\p{U}\+\q{U}\-1)}{u_2}$ is
a prefix of $v$ for some suffix $s_1$ of $u_1$ and some $i\geq 1$.

\end{my_itemize}
\item[$(b)$] If $\send{v_{[1]}}\leq \send{u_{[1]}}$, then $\sbig{v_{[1]}}<R_1({\cal U})$ and either $(a_1)$ or $(a_2)$ holds.
\end{my_itemize}
\end{lemma}

\singlespacing
Definition~\ref{def-rels} formalizes the types of relationship implied by 
Lemma~\ref{v-cases}. 

\onehalfspacing
\begin{definition}
\label{def-rels}
We say that an \fsds\ $\cal V$ is a \emph{mate} of an \fsds\ $\cal U$ in a string $x$,
if $\sbig{{\cal U}}<\sbig{{\cal V}}$.

\vspace{-5pt}
\begin{my_enumerate}
\leftskip=-10pt
\item $\cal V$ is an \emph{\amate} of $\cal U$ if $\sbig{{\cal V}}\leq \sbig{{\cal U}}\+$\lcp{u} and  $\cal V$ is a right cyclic shift of $\cal U$.
\item $\cal V$ is a \emph{\bmate} of $\cal U$ if $\sbig{{\cal V}} < \send{v_{[1]}}<\send{u_{[1]}}$ and
$v={{\widehat u}_1}\np{i}{{\widehat u}_2}$ for some $1<i<\p{U}$ where
${{\widehat u}_2}$ is a non-trivial prefix of ${{\widehat u}_1}$ and where ${{\widehat u}_1}$ respective ${{\widehat u}_2}$ is a cyclic
shift of $u_1$ respective $u_2$ in the
same direction by the same number of positions, and $V^2$ is a 
right cyclic shift of $U^2$ by $\sbig{{\cal V}}\-\sbig{{\cal U}}$ positions.
\item $\cal V$ is a \emph{\gmate} of $\cal U$  if 
$\sbig{{\cal V}}<\sbig{{\cal U}}\+\p{U}|u_1|$ and $|v|=|U|$.
\item $\cal V$ is a \emph{\dmate} of $\cal U$ if $\sbig{{\cal V}}<R_1({\cal U})$
and $|v|>|U|$ and either\newline
$s_1\o{u}_2u_2{u_1}^{(\p{U}\+\q{U}\-1)}{u_2}$ is a non-trivial
prefix 
of $v$ for some suffix $s_1$ of $u_2$, or
$s_1{u_1}^iu_2{u_1}^{(\p{U}\+\q{U}\-1)}{u_2}$ 
is a non-trivial
prefix of $v$ for some $s_1$ suffix of $u_1$ and some $i\geq 1$.
\item $\cal V$ is an \emph{\emate} of $\cal U$ if $R_1({\cal U})\leq \sbig{{\cal V}}$.
If, in addition, $\send{u_{[1]}}<\sbig{{\cal V}}$, we will call $\cal V$ a \emph{\eemate}.
\end{my_enumerate}
\end{definition}

\vskip -10pt
\singlespacing
Note that Definition~\ref{def-rels} implies that an \amate\ of an \amate\ of $\cal U$
is an \amate\ of $\cal U$; an \amate\ of a \bmate\ of $\cal U$ is \bmate\ of $\cal U$;
a \bmate\ of a \bmate\ of $\cal U$ is a \bmate\ of $\cal U$; if $\cal V$ is \bmate\
of $\cal U$, then $|U|=|V|$, $V={{\widehat u}_1}\np{i}{{\widehat u}_2}{{\widehat u}_1}^{(\p{U}\+\q{U}\-i)}$, and $\p{U}\-\q{U}\geq 2$ since $i\geq \p{U}\+\q{U}\-i$. If $\cal V$ is a \gmate\ of $\cal U$, then $v^2$ is right cyclic shift 
of $U^2$.

\begin{lemma}
\label{vv-cases}
Let $x$ be a string starting with an \fsds\ $\cal U$. Let $\cal V$ be an \fsds\ 
with $\sbig{{\cal U}} < \sbig{{\cal V}}$, then either

\vspace{-10pt}
\begin{my_itemize}
\leftskip=-7pt
\item[$(a)$] ${\sbig{\cal V}}<R_1(\cal U)$,  in which case either

\vspace{-5pt}
\begin{my_itemize}
\leftskip=-14pt
\item[$(a_1)$] $\cal V$ is an \amate\ of $\cal U$, or
\item[$(a_2)$] $\cal V$ is a \bmate\ of $\cal U$ and $\p{U}> \q{U}\+1$,  or
\item[$(a_3)$] $\cal V$ is a \gmate\ of $\cal U$, or
\item[$(a_4)$]  $\cal V$ is a \dmate\ of $\cal U$,
\end{my_itemize}
\item[] or
\item[$(b)$] $R_1(\cal U)\leq {\sbig{\cal V}}$, then
\begin{my_itemize}
\leftskip=-14pt
\item[$(b_1)$] $\cal V$ is a \emate\ of $\cal U$ and $\send{v_{[1]}}>\send{u_{[1]}}$.
\end{my_itemize}
\end{my_itemize}
\end{lemma}

\noindent
The rather technical proofs of Lemmas~\ref{v-cases} and~\ref{vv-cases} are given, respectively, 
in Sections~\ref{proof-v-cases} and~\ref{proof-vV-cases}.

\subsection{Some properties of \gmate{s}}

\medskip
\noindent
Let an \fsds\ $\cal V$ be a \gmate\ of an \fsds\ $\cal U$. Then 
$v = s_2{u_1}^{\p{U}\-t\-1}u_2{u_1}^{\q{U}\+t}s_1$  or
$v = {u_1}^{\p{U}\-t}u_2{u_1}^{\q{U}\+t}$ for some $\p{U}\-t\geq 1$ and
some $s_1$, $s_2$ so that  $s_1s_2=u_1$. Let us define a \emph{type} of  $\cal V$:

\smallskip
\noindent
$type({\cal V}) = \begin{cases} 
(\p{U}\-t,\q{U}\+t) &\mbox{if }v = {u_1}^{\p{U}\-t}u_2{u_1}^{\q{U}\+t} \\
(\p{U}\-t,\q{U}\+t) &\mbox{if }s_2{u_1}^{\p{U}\-t\-1}u_2{u_1}^{\q{U}\+t}s_1\ \text{and}\\
\ & |s_1|\leq |u_1|\-lcs(u_1,\c{u}_1) \\
(\p{U}\-t-1,\q{U}\+t\+1) & \mbox{otherwise}. 
\end{cases}$

\medskip
\noindent
Though we do not know exactly what $V^2$ is like, we can still 
determine some of its properties.
\begin{lemma}
\label{gamma}
Let an \fsds\ $\cal V$ be a \gmate\ of an \fsds\ $\cal U$ of type $(p,q)$ where $p, q\geq 2$
and $p\+q\geq 4$. Then  $\p{V}=\q{V}$ and $|v_2|\leq min(p,q)|u_1|$.
Moreover, either $|v_2|<|u_1|$ or there is a factor $({u_1}^qu_2)({u_1}^qu_2)$ in $V^2$.
\end{lemma}

\begin{proof}
Let us first assume that $v^2=[{u_1}^pu_2{u_1}^q][{u_1}^pu_2{u_1}^q]$.

\vspace{-10pt}
\begin{my_itemize}
\leftskip=-5pt
\item[$(a)$] Let $p\geq q$.\\
By Lemma~\ref{syncpr}, the leftmost possible beginning of $V_{[2]}$ can be at\linebreak
$|{u_1}^pu_2{u_1}^{p\+q}u_2|\+1$ and so ${u_1}^pu_2$ is a prefix 
of ${v_1}^\q{V}$ and $v_2$ is a factor of ${u_1}^q$.
First we prove that $|v_1|> (p\-1)|u_1|$:

\smallskip
\leftskip=7pt
\noindent
Assume that $|v_1|\leq (p\-1)|u_1|$.
Then ${u_1}^p$ contains a factor of size $|v_1|\+|u_1|$ and the same factor is also
contained in ${v_1}^\q{V}$ as ${u_1}^pu_2$ is a prefix 
of ${v_1}^\q{V}$. If $\q{V}\geq 2$, then by
Lemma~\ref{comfactor}, $u_1=v_1$ and so 
${u_1}^pu_2$ is a prefix of ${u_1}^{\q{V}}$ and
thus ${u_1}^pu_2u_1$ is a prefix of ${u_1}^{\q{V}\+1}$, 
which contradicts Lemma~\ref{syncpr}.
Therefore $\q{V}=1$ and so $|v_1|\geq p|u_1|\+|u_2|>(p\-1)|u_1|$,
a contradiction with the assumption.

\smallskip
\leftskip=0pt
\noindent
Hence $|v_1|> (p\-1)|u_1| \geq q|u_1|$ and since $v_2$ is a factor in ${u_1}^q$,
$\p{V}=\q{V}$.\\
If $V_{[2]}$ begins even more to the right, this makes $v_2$ smaller
and ${v_1}^\q{V}$ bigger, thus the same argument can be applied.
\item[$(b)$] Let $p<q$\\
By Lemma~\ref{syncpr} the leftmost possible beginning of $V_{[2]}$ can be at\linebreak
$|{u_1}^pu_2{u_1}^{p\+q}u_2{u_1}^{q\-p}|\+1$ and so ${u_1}^pu_2{u_1}^{q\-p}$ is a prefix 
of ${v_1}^\q{V}$ and $v_2$ is a factor of ${u_1}^{p}$.
Let $r=max(p,q\-p)$. First we prove that $|v_1|> (r\-1)|u_1|$:

\smallskip
\leftskip=7pt
\noindent
Assume that $|v_1|\leq (r\-1)|u_1|$.
Then either ${u_1}^p$ or ${u_1}^{q\-p}$ contains a factor of size $|v_1|\+|u_1|$ and the same factor is also contained in ${v_1}^\q{V}$ as ${u_1}^pu_2{u_1}^{q\-p}$ is a prefix 
of ${v_1}^\q{V}$. If $\q{V}\geq 2$, then by Lemma~\ref{comfactor}, $u_1=v_1$ and so 
${u_1}^pu_2{u_1}^{q\-p}$ is a prefix of ${u_1}^{\q{V}}$, which contradicts Lemma~\ref{syncpr}.
Therefore $\q{V}=1$ and so $|v_1|\geq q|u_1|\+|u_2|>(r\-1)|u_1|$,
a contradiction with the assumption.

\smallskip
\leftskip=0pt
\noindent
Hence $|v_1|> (r\-1)|u_1| \geq p|u_1|$ and since $v_2$ is a factor in ${u_1}^p$,
$\p{V}=\q{V}$.\\
If $V_{[2]}$ begins even more to the right, this makes $v_2$ smaller
and ${v_1}^\q{V}$ bigger, thus the same argument can be applied.
\end{my_itemize}

\smallskip
\noindent
Let us thus assume that $v^2=[s_2{u_1}^{p\-1}u_2{u_1}^qs_1][s_2{u_1}^{p\-1}u_2{u_1}^qs_2]$
and $|s_1|\leq |u_1|\-$\lcs{u}. Then $|s_2|>$ \lcs{u}.

\vspace{-5pt}
\begin{my_itemize}
\leftskip=-5pt
\item[$(a)$] Let $p\geq q$.\\
By Lemma~\ref{syncpr}, the leftmost possible beginning of $V_{[2]}$ can be at\linebreak
$|s_2{u_1}^{p\-1}u_2{u_1}^{p\+q}u_2s_1|\+1$. If it started to the left of this point,
by Lemma~\ref{syncpr}, $s_2$ would have to be a suffix of $u_1u_2$ and so $s_2$
would be a common suffix of $u_1$ and $\c{u}_1$, and so $|s_2|\leq$
\lcs{u}, a contradiction.
Therefore the same arguments as in the case  
$v^2=[{u_1}^pu_2{u_1}^q][{u_1}^pu_2{u_1}^q]$
can be applied.
\item[$(b)$] Let $p<q$\\
By Lemma~\ref{syncpr} and by $|s_2|>$ \lcs{u}, 
the leftmost possible beginning of $V_{[2]}$ can be at $|s_2{u_1}^pu_2{u_1}^{p\+q}u_2{u_1}^{q\-p}s_1|\+1$. Again, if it started to the 
left of this point, by Lemma~\ref{syncpr}, $s_2$ would have to be a suffix of $u_1u_2$ and so 
$s_2$ would be a common suffix of $u_1$ and $\c{u}_1$, and so 
$|s_2|\leq$ \lcs{u}, a contradiction.
Therefore, the same arguments  as in the case  $v^2=[{u_1}^pu_2{u_1}^q][{u_1}^pu_2{u_1}^q]$ can be applied.
\end{my_itemize}
If $|v_2|\geq |u_1|$, then a prefix of $V_{[2]}$ must align with the last $u_1$
of ${u_1}^pu_2{u_1}^{q\+p}u_2{u_1}^q$ and so ${u_1}^pu_2{u_1}^{q\+p}u_2{u_1}^q$ 
is extended for sure by another $u_2$, i.e. $V^2$ contains a factor
${u_1}^qu_2{u_1}^qu_2$.
\end{proof}

\subsection{Some properties of \emate{s} of $\cal U$}

\medskip
\begin{lemma}
\label{ge}
Let ${\cal U}, {\cal V}, {\cal W}$ be \fsds{s} so that 
$\sbig{{\cal U}}<\sbig{{\cal V}}<\sbig{{\cal W}}$. Let $\cal V$ be a \gmate\ of $\cal U$ 
of type $(\p{U}\-t,\q{U}\+t)$, $2\leq p\-t$ and $2\leq q\+t$, 
and let $\cal W$ be an \emate\ but not a \eemate\ of $\cal V$. Then $G({\cal U},{\cal W})\geq t|u_1|$ and $T({\cal U},{\cal W})\geq (\p{U}\+\q{U})|u_1|$.
\end{lemma}

\begin{proof}
The position of $v^2$ is:

\smallskip
\noindent
${u_1}^ts_1{\big[}s_2{u_1}^{\p{U}\-t\-1}u_2{u_1}^{\q{U}\+t}s_1{\big]}{\big[}s_2{u_1}^{\p{U}\-t\-1}u_2{u_1}^{\q{U}\+t}s_1{\big]}$. Since $\cal V$ is 
a \gmate\ of $\cal U$, by Lemma~\ref{gamma} $\p{V}=\q{V}$ and so $\cal V$
cannot have a \bmate, see Lemma~\ref{vv-cases}. Thus $w_{[1]}$ must end past
the end of $v_{[1]}$ and thus by Lemma~\ref{syncpr}, $|w|\geq |v|$. Therefore,
$G\geq t|u_1|$ and $T\geq (\p{U}\+\q{U})|u_1|$.
\end{proof}

\begin{lemma}
\label{supere}
Let $\cal V$ be a \eemate\ of $\cal U$. Then either

\vspace{-10pt}
\begin{my_itemize}
\leftskip=-5pt
\item[$(a)$] $G({\cal U},{\cal V})\geq (2\p{U}\+\q{U}\-3)|u_1|\+2|u_2|$ and\newline
$T({\cal U},{\cal V})\geq (\p{U}\+\q{U}\-2)|u_1|\+|u_2|$, or
\item[$(b)$] $G({\cal U},{\cal V})\geq \p{U}|u_1|\+|u_2|$ and
$T({\cal U},{\cal V})\geq (\p{U}\+\q{U}\-1)|u_1|\+|u_2|$.
\end{my_itemize}
\end{lemma}

\begin{proof}
If ${v^2}$ were a factor of ${u_1}^{\p{U}\+\q{U}\-1}u_2$, then there
would be a farther copy of $v^2$ in ${u_1}^{\p{U}\+\q{U}}u_2$ -- just starting $|u_1|$ positions to the right, which is a contradiction as $v^2$ must be a rightmost occurrence.
Hence $\send{v^2}>|{u_1}^\p{U}u_2{u_1}^{\p{U}\+\q{U}\-1}u_2|$.\\
Let us assume that $v_{[1]}$ is a factor in ${u_1}^{\p{U}}u_2{u_1}^{(\p{U}\+\q{U}\-1)}u_2$.\\
Then ${u_1}^{(\p{U}\+\q{U})}u_2$ and $v^2$ both contain a common factor 
of size $|v|\+|u_1|$, and thus by Lemma~\ref{comfactor}, $v={v_1}^k$ for some
conjugate $v_1$ of $u_1$ and some $k\geq 1$.
If $k=1$, then $\sbig{v_{[1]}}\geq |{u_1}^\p{U}u_2{u_1}^{(\p{U}\+\q{U}\-3)}u_2|$ 
and so $G\geq |{u_1}^\p{U}u_2{u_1}^{(\p{U}\+\q{U}\-3)}u_2|$. Moreover
$\sbig{v_{[2]}}=\sbig{v_{[1]}}\+|u_1|$ and so
$T\geq$\linebreak
$|{u_1}^{(\p{U}\+\q{U}\-2)}u_2|$, i.e. $(a)$ holds.\\
\ \\
Let us assume that $k\geq 2$. We will discuss two cases: 
\begin{my_itemize}
\leftskip=-7pt
\item[$(i)$] $v_{[1]}$ starts in $\o{u}_2$ and ends in $\o{u}_2$
Then there are $s_1s_2=\o{u}_2$ so that $v=(s_2u_2s_1)^k$ and so
that $v^2s_2$ is a suffix of ${u_1}^\p{U}u_2{u_1}^{(\p{U}\+\q{U})}$.
\begin{my_itemize}
\leftskip=-20pt
\item[$(i_1)$] Let $|s_2|\leq lcs(u_1,\c{u}_1)$.\\
Then we can assume without loss of generality that $v={u_1}^k$ as otherwise we can cyclically shift
the whole structure $|s_2|$ positions to the left. By Lemma~\ref{nonprimitiveds},
$V={u_1}^{2k\-1}t_1$ for some non-trivial proper prefix $t_1$ of $u_1$. Let $t_1t_2=u_1$.
Then the prefix ${u_1}^3$ of $V_{[2]}$ must align by Lemma~\ref{syncpr} with $t_2u_1u_1$ and
hence $t_2u_2=u_1$. Therefore $|t_2|=|\o{u}_2|$ and since $t_2$ is a
suffix of $u_1=u_2\o{u}_2$, in fact $t_2=\o{u}_2$, Hence
$u_1=\o{u}_2u_2$, a contradiction.
\item[$(i_2)$] Let $|s_2|>lcs(u_1,\c{u}_1)$.\\
Then by Lemma~\ref{nonprimitiveds}, $V=(s_2u_2s_1)^{2k\-1}t_1$ where $t_1$ is
a non-trivial proper prefix of $s_2u_2s_1$. Let $t_1t_2=s_2u_2s_1$. Then
the prefix $(s_2u_2s_1)^3$ of $V_{[2]}$ must align by Lemma~\ref{syncpr} with
$t_2u_2u_2s_1s_2u_2s_1s_2u_2$ and so either
$t_2u_2=s_2$ or $t_2u_2=s_su_2s_1s_2$. In either case,
$s_2$ is a suffix of $t_2u_2$ and since $s_2$ is a suffix if $\o{u}_2$,
$s_2$ is both a suffix of $u_1$ and of $\c{u}_1$. Hence
$|s_2|\leq lcs(u_1,\c{u}_1)$, a contradiction.
\end{my_itemize}
\item[$(ii)$]  $v_{[1]}$ starts in $u_2$ and ends in $u_2$.\\
Then there are $s_1s_2=u_2$ so that $v=(s_2\o{u}_2s_1)^k$ and so
that $v^2s_2$ is a suffix of ${u_1}^\p{U}u_2{u_1}^{(\p{U}\+\q{U})}u_2$.
\begin{my_itemize}
\leftskip=-20pt
\item[$(ii_1)$] Let $|s_2|\leq lcs(u_1,\c{u}_1)$.\\
Then without loss of generality we can assume $v=(\o{u}_2u_2)^k$ and $v^2$ is a suffix of ${u_1}^\p{U}u_2{u_1}^{(\p{U}\+\q{U})}u_2$ as otherwise we could cyclically
shift the whole structure $|s_2|$ positions to the left.
Then a suffix\newline
$(\o{u}_2u_2)(\o{u}_2u_2)(\o{u}_2u_2)(\o{u}_2u_2)$
of $v^2$  must align with\newline
$(\o{u}_2u_2)(\o{u}_2u_2)(\o{u}_2u_2)(u_2\o{u}_2)(u_2\o{u}_2)$
giving $\o{u}_2u_2=u_2\o{u}_2$, a contradiction.
\item[$(ii_2)$] Let  $|s_2| > lcs(u_1,\c{u}_1)$.\\
Then $v=(s_2\o{u}_2s_1)^k$ and by Lemma~\ref{nonprimitiveds},
$V=(s_2\o{u}_2s_1)^{2k\-1}t_1$ and $t_1t_2=s_2\o{u}_2s_1$. Then
a prefix $(s_2\o{u}_2s_1)^3$ of $V_{[2]}$ must align by Lemma~\ref{syncpr} with
$t_2s_1s_2\o{u}_2s_1s_2\o{u}_2$ and so
$t_2=s_2\o{u}_2$. Since $t_1t_2=s_2\o{u}_2s_1$, then
$t_1t_2s_2=s_2\o{u}_2s_1s_2=s_2\o{u}_2u_2$, i.e.
$t_1t_2s_2=s_2\c{u}_1$ and so $s_2$ is both a suffix of
$\c{u}_1$ and a suffix of $u_2$ and hence of $u_1$, and so
$|s_2|\leq lcs(u_1,\c{u}_1)$, a contradiction.
\end{my_itemize}
\end{my_itemize}
Considering the end of $v^2$ in the next $\o{u}_2$ will yield a contradiction using the same argumentation as for $(i)$, and 
considering the end of $v^2$ in the next $u_2$
will yield a contradiction using the same argumentation as for $(ii)$.\\
Thus, the only remaining case is when $v_{[1]}$ is not a factor in\newline ${u_1}^{\p{U}}u_2{u_1}^{(\p{U}\+\q{U}\-1)}u_2$, i.e.
$\send{v_{[1]}}>{u_1}^{\p{U}}u_2{u_1}^{(\p{U}\+\q{U}\-1)}u_2$ and
so $G\geq |{u_1}^\p{U}u_2|$ and
$T\geq |{u_1}^{(\p{U}\+\q{U}\-1)}u_2|$, i.e. case $(b)$ holds.
\end{proof}

\section{An upper bound for the number of {\fsds}s}
In this section, we only consider strings containing at least one {\fsds}. Let $\delta({x})$ denote the number of \fsds{s} in $x$.
We  prove by induction that $\delta({x})\leq \frac{7}{8}|x|\-\frac{3}{8}|u|$  where $u$ is the generator of the shorter square of the 
first \fsds\ in $x$. 
We first need to investigate the relationship between two \fsds{s} of $x$ as the induction hypothesis is 
applied to the substring starting at some \fsds\ and extended to the
string starting with the first \fsds. 

\begin{lemma}
\label{GapTail}
Let $x$ be a string starting with an \fsds\ $\cal U$ and let $\cal V$ be another \fsds\ of $x$ with $\send{u_{[1]}}\leq\send{v_{[1]}}$. 
Let $x'$ be the suffix of $x$ starting at the same position as $\cal V$.
Let $d$ be the number of \fsds{s} 
between $\cal U$ and $\cal V$ including $\cal U$ but not including $\cal V$. 
Then, $\delta({x'})\leq \frac{5}{6}|x'|\-\frac{1}{3}|v|$ implies $\delta({x})\leq \frac{5}{6}|x|\-\frac{1}{3}|u|\+d\-\frac{1}{2}|G({\cal U},{\cal V})|\-\frac{1}{3}|T({\cal U},{\cal V})|$.
\end{lemma}

\begin{proof}
\onehalfspacing
As $|G|\+|v|=|u|\+|T|$, we have $\-\frac{1}{3}|v|=\-\frac{1}{3}|u|\-\frac{1}{3}|T|\+\frac{1}{3}|G|$.
Thus, $\delta({x})\leq d\+\delta({x'})\leq d\+\frac{5}{6}|x'|\-\frac{1}{3}|v|=
d\+\frac{5}{6}|x'|\-\frac{1}{3}|u|\-\frac{1}{3}|T|\+\frac{1}{3}|G|$. 
Thus, $\delta({x})\leq \frac{5}{6}(|x'|\+|G|)\-\frac{1}{3}|u| \+d\-\frac{5}{6}|G|\+\frac{1}{3}|G|\-\frac{1}{3}|T|
=\frac{5}{6}|x|\-\frac{1}{3}|u| \+d\-\frac{1}{2}|G|\-\frac{1}{3}|T|$ since\linebreak
$|x|=|x'|\+|G|$.
\end{proof}

\vskip -20pt
\singlespacing
Lemma~\ref{GapTail} yields a straightforward induction step whenever $\frac{1}{2}|G|\+\frac{1}{3}|T|\geq d$. 
By Lemma~\ref{vv-cases}, this condition always holds except for the two cases: either ${\cal V}$ is a right cyclic shift of $\cal U$ by 1 position and hence an \amate\ of ${\cal U}$, since then $\frac{1}{2}|G|\+\frac{1}{3}|T|=
\frac{1}{2}+\frac{1}{3}=\frac{5}{6}\ne 1$, or 
${\cal V}$ is a \bmate\ of ${\cal U}$ and such that 
$\send{v_{[1]}}<\send{u_{[1]}}$ -- hence Lemma~\ref{GapTail} is not
applicable. Therefore the whole
group of \amate{s} and \bmate{s} of $\cal U$ must be dealt together in the
induction
rather than carrying it from one \fsds\ to another. Since a \gmate\ of $\cal U$
does not provide a sufficiently large tail to offset all of the \amate{s} and
\bmate{s} of $\cal U$ preceding it, we have to include them in the special treatment as
well -- this is all precisely defined and explained in Section~\ref{abg}.
First we need to strengthens the bound on the length of the maximal right
cyclic shift of $\cal U$ when $\p{U}=\q{U}$.

\begin{figure}
\begin{center}
\includegraphics[scale=1]{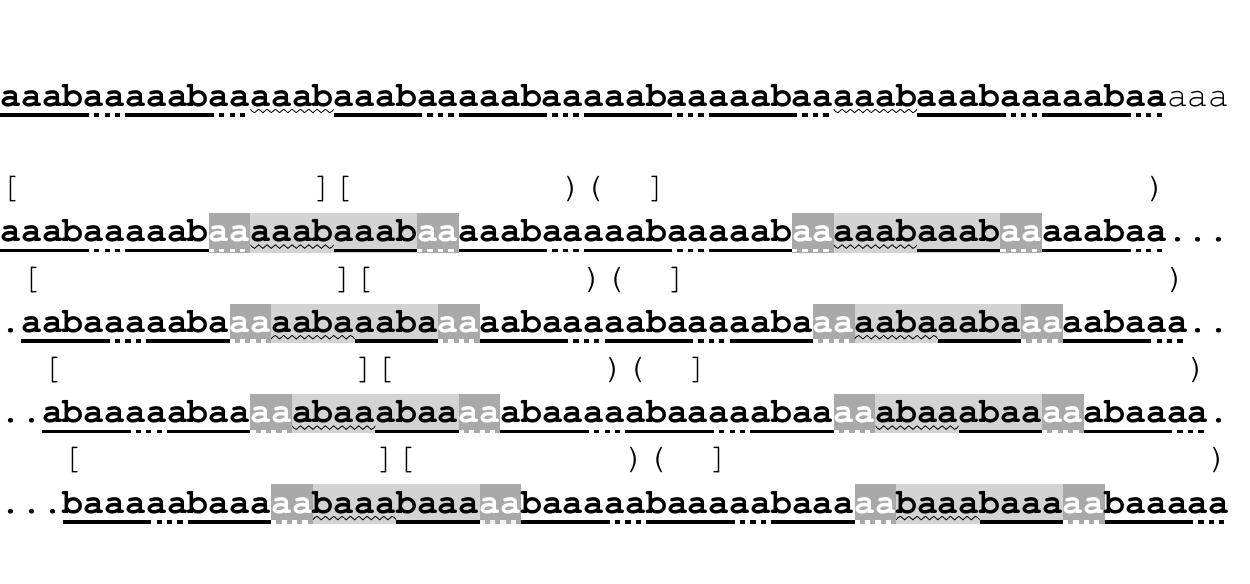}
\caption{Example of an $\alpha$-family of $\cal U$ with {\scriptsize $\,{\cal U}(1)={\cal U}(2)$}}
\end{center}
\end{figure}

\begin{lemma}
\label{rot1}
Let $x$ be a string starting with an \fsds\ ${\cal U}$  such that $\p{U}=\q{U}$,
i.e. $x$ = $U^2y$ for some, possibly empty, $y$, then 
$lcp(u,y)<min\{|y|,|u_2|\}$.
\end{lemma}

\begin{proof}
Lemma~\ref{rot1} trivially holds if $|y|\leq |u_2|$. Let us assume $|y|> |u_2|$ and 
$lcp(u,y) \geq |u_2|$. Let $e=\p{U}=\q{U}$. 
Then $x = U^2u_2z$ for some $z$ and thus,  $x={u_1}^eu_2{u_1}^e{\underline{{u_1}^eu_2{u_1}^eu_2}}z$, i.e. there is a
farther occurrence of $u^2$ (underlined), a contradiction.
\end{proof}

\subsection{Handling $\alpha$, $\beta$, and $\gamma$ mates}
\label{abg}

\smallskip
\noindent
The basic unit for our induction is what we call \emph{$\cal U$ family},
or equivalently \emph{family of $\cal U$}, 
which is presented in Definition~\ref{family}.

\begin{definition}
\label{family}
\noindent Let $x$ be a string starting with an \fsds\ $\cal U$. 
If all \fsds{s} in $x$ are \amate{s} of $\cal U$, then \emph{$\cal U$ family} consists of $\cal U$ and all its \amate{s}. Otherwise, let $\cal V$ be the rightmost 
\fsds\ that is not an \amate\ of $\cal U$. If $\cal V$ is not a \bmate\ of $\cal U$, then \emph{$\cal U$ family} consists of $\cal U$ and its \amate{s}. In all other
cases \emph{$\cal U$ family} consists of $\cal U$ and all its \amate{s}, \bmate{s}, and \gmate{s}.
\end{definition}

In the following sections we discuss the possible formats and sizes of $\cal U$ family.

\subsubsection{The case $\cal U$ family consists only of \amate{s}}
\label{alpha-family}

\smallskip
\noindent
We call such a family an $\alpha$-family. The family is either followed by no
other \fsds, or it is followed by a \gmate, a \dmate, or an \emate. If it were followed
by a \bmate, it would be an \bfam\ or an \gfam\ discussed in the following 
sections.

If $\p{U}=\q{U}$, then 
$u^2$ can be non-trivially cyclically shifted to the right at most  $|u_2|\-1$ times by Lemma~\ref{rot1},  and so the size of the $\cal U$ family is at most $|u_2|$. Since $U^2$ must be 
non-trivially cyclically shifted as well, $U^2$ must be followed by a prefix of $u_2$ of the same size.
See Figure~1 for an illustration of an $\alpha$-family where
$u_1=aaabaa$, $u_2=aaab$, $\o{u}_2=aa$, $\p{U}=\q{U}=2$. 
The solid underline  indicates $u_2$, and the dotted underline 
indicates $\o{u}_2$.
The extension of $U^2$ is the final suffix not in bold. The \fsds\
$\cal U$ can be non-trivially  cyclically shifted to the right by 
\lcp{u} = $lcp(u_2\o{u}_2,\o{u}_2u_2)=
lcp(aaabaa,aaaaab)=3$ as the extension of $U^2$ is $aaa$ which
is a prefix of $u_2$ of size 3. Thus, the family has a size of 4 which
equals $|u_2|$.  Note that if the string were extended by the next symbol of $u_2$ which is $b$,  $\cal U$ would cease to 
be an \fsds\ as its shorter square would have a farther occurrence.

If $\p{U}>\q{U}$, then
by Lemma~\ref{rot2}, $u^2$ can be non-trivially cyclically shifted at most $|u_1|\-2$ times,
therefore, the size of the $\cal U$ family  is at most $|u_1|\-1$. Since $U^2$ must be non-trivially 
cyclically shifted as well, $U^2$ must be followed by a prefix of $u_1$ of the same size.
See Figure~3 for an illustration where
$u_2=aaab$, $\o{u}_2=aa$, $\p{U}=2$, and $\q{U}=1$. The extension of $U^2$ is the final suffix not in bold.
Therefore $u_1=aaabaa$,  $\c{u}_1=aaaaab$, \lcp{u} = 3, and 
\lcs{u} = 0. Thus,
 $\cal U$ can be non-trivially cyclically shifted 3 times to the right  as the extension of $U^2$ is $aaa$ which is a prefix of $u_1$ of size 3, and not at all to the left.
The size of the  family is 4 and equals $|u_1|\-2$.  Note that if we extend the string by the next symbol of $u_1$, which is $b$, we do not gain 
yet another \fsds\ since the maximal shift of $u^2$ to the right is exhausted and
so only $U^2$ would be cyclically shifted.

\begin{figure}
\begin{center}
\includegraphics[scale=1]{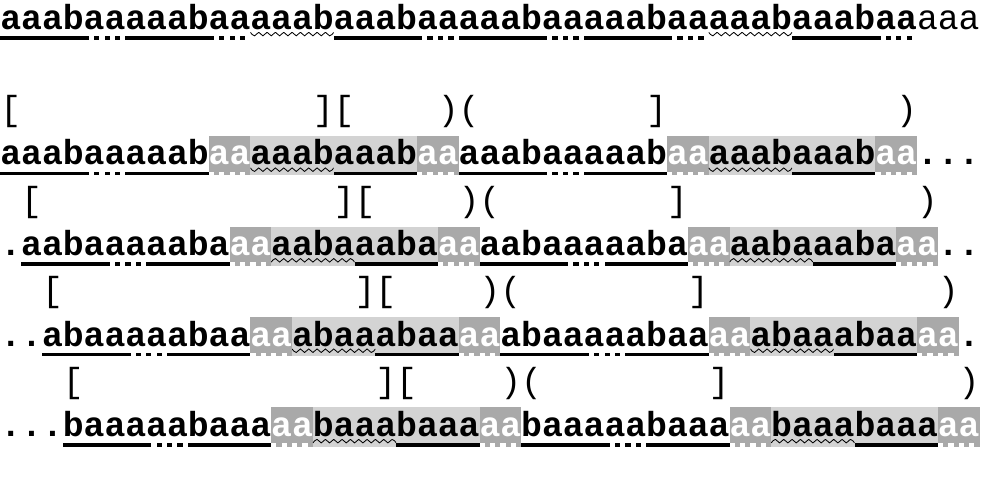}
\caption{Example of an $\alpha$-family of $\cal U$ with {\scriptsize $\,{\cal U}(1)>{\cal U}(2)$}}
\end{center}
\end{figure}

\begin{claim}
\label{bottomless-alpha}
Let $x$ be a string starting with an \afam\ of an \fsds\ $\cal U$ with no additional
\fsds{s} in $x$, then $\delta(x)\leq\frac{5}{6}{|x|}\-\frac{1}{3}{|u|}$.
\end{claim}

\onehalfspacing
\begin{proof}
Let $f$ be the size of the $\cal U$-family. It follows that  $f< |u_1|$.
Note that $|u|=\p{U}|u_1|\+|u_2|$.
Since $|x|\geq |U^2|\+f=2(\p{U}\+\q{U})|u_1|\+2|u_2|\+f$, we get 
$\frac{5}{6}{|x|}\-\frac{1}{3}{|u|}\geq
\frac{5}{6}(2\p{U}\+\q{U})|u_1|\+\frac{5}{6}2|u_2|\-\frac{2}{3}p|u_1|\-\frac{1}{3}|u_2|= 
\frac{6\p{U}\+5\q{U}}{6}|u_1|\+$\newline
$\frac{8}{6}|u_2|>\frac{11}{6}|u_1|> f=\delta(x)$.
\end{proof}

\begin{claim}
\label{bottom-alpha}
Let $x$ be a string starting with an \afam\ of an \fsds\ $\cal U$. Let $\cal V$ be the first \fsds\ that is not a member of the $\cal U$ family.
If $\delta({x'})\leq \frac{5}{6}{|x'|}\-\frac{1}{3}{|v|}$ where $x'$ is a suffix of $x$ starting
at the same position as $\cal V$, then
$\delta({x})\leq \frac{5}{6}{|x|}\-\frac{1}{3}{|u|}$.
\end{claim}

\vskip -10pt
\singlespacing
\begin{proof}
Let $f$ be the size of the $\cal U$ family, then $f\leq |u_1|$.
Let $\cal W$ be the last member of the \afam\ of $\cal U$. Note that ${\cal W}={\cal U}$ when the ${\cal U}$ family consists only of ${\cal U}$.
We apply Lemma~\ref{vv-cases} to $\cal W$ and $\cal V$: since $\cal V$
is neither an \amate\  nor a \bmate\ of $\cal W$, then either it is a \gmate\ or a \dmate, or
an \emate\ of $\cal W$. If it is a \gmate\ or a \dmate, then $|v|\geq |W|$ and so 
the size of the tail between
$\cal W$ and $\cal V$ is at least $\q{W}|u_1|$. Since $\q{W}=\q{U}\geq 1$,
the size of the tail is at least $|u_1|$.
Therefore, the size of the gap $G$
between $\cal U$ and $\cal V$ is at least $f$, the size of the tail $T$
between $\cal U$ and $\cal V$ is at least $f\+|u_1|\geq 2f$. 
Therefore, $\frac{1}{2}|G|\+\frac{1}{3}|T|\geq
\frac{1}{2}f\+\frac{1}{3}2f=\frac{7}{6}f>f$. 
If $\cal V$ is an \emate\ of $\cal W$, then
the gap between $\cal W$ and $\cal V$ is at least $u_1$ and the
tail exists. Hence, the gap between $\cal U$ and $\cal V$ is at least
$f\+|u_1|\geq 2f$ and the tail exists. Therefore,
$\frac{1}{2}|G|\+\frac{1}{3}|T|\geq \frac{1}{2}2f=f$.
By Lemma~\ref{GapTail},
$\delta(x)\leq \frac{5}{6}{|x|}\-\frac{1}{3}{|u|}$.
\end{proof}

\subsubsection{The case $\cal U$ family consists of both \amate{s} and \bmate{s} with no \gmate{s}}
\label{family-beta}

\smallskip
\noindent
A $\cal U$ family consisting entirely of \amate{s} and \bmate{s} of $\cal U$ 
is called an \bfam\ and has the following structure:
\begin{my_itemize}
\leftskip=-10pt
\item[.] The first so-called \emph{\aseg}
consists of $\cal U$ and possibly its right cyclic shifts,
i.e. its \amate{s}. The size of the segment is $\leq$ \lcp{u} $\leq |u_1|\-2$, see Lemma~\ref{rot2}. All the \fsds{s} in this segments have 
the first exponent equal to $\p{U}$ and the second exponent
equal to $\q{U}$, thus we say that the type of the segment is $(\p{U},\q{U})$.

\item[.] Then there must be a \bmate\ of $\cal U$ and possibly its right cyclic shifts.
 All the \fsds{s} in the segment have the first exponent equal to 
$\p{U}\-i_1$ and the second exponent equal to $\q{U}\+i_1$ for some 
$1\leq i_1<(\p{U}\-\q{U})/2$, thus we say that the type of the segment is
$(\p{U}\-i_1,\q{U}\+i_1)$. This so-called \emph{\bseg} has size $\leq$ \lcp{u} 
$\leq |u_1|\-2$ if $\p{U}\-i_1>\q{U}\+i_1$, see Lemma~\ref{rot2},
or $\leq |u_2|\-1\leq |u_1|\-2$ if $\p{U}\-i_1=\q{U}\+i_1$. 

\item[.] Then there may be another \bseg\
of type $(\p{U}\-i_2,\q{U}\+i_2)$ for some $1\leq i_1<i_2<(\p{U}\-\q{U})/2$, etc. 
\item[.] Either there is no other \fsds\ in $x$, or the first \fsds\ after the last
member of the last \bseg\ must be either a  \dmate\ or an \emate\ of $\cal U$,
since if it were a \gmate, then the $\cal U$ family would be an \gfam\ discussed
in the following section.
\end{my_itemize}

\onehalfspacing
\noindent
There may be $t$ such \bseg{s} where $2t\leq \p{U}\-\q{U}$.
Let the last \bseg\ be of type $(\p{U}\-t,\q{U}\+t)$.  If 
$\p{U}\-t=\q{U}\+t$ (which implies that $\p{U}$ is odd and $\p{U}-\q{U}$ is even),
then $2t=\p{U}\-\q{U}$ and there are $\leq (\p{U}-\q{U})/2$ segments of
size $\leq |u_1|$ and 1 segment of size $\leq |u_2|$ and so the size of
the family $f\leq \frac{\p{U}\-\q{U}}{2}|u_1|+|u_2|$. If $\p{U}\-t>\q{U}\+t$,
there are two cases, either $\q{U}=1$ and then 
$f\leq {\big\lceil}\frac{\p{U}\-\q{U}}{2}{\big\rceil}|u_1|$,
or $\q{U}>1$ and $f\leq \frac{\p{U}\-\q{U}}{2}|u_1|$.

\vskip -10pt
\singlespacing
See Figure~4 for an illustration of an \bfam\ where
$u_2=aaab$, $\o{u}_2=aa$, $\p{U}=5$, and $\q{U}=1$. 
The configuration of square brackets \verb+[  ][  ]+ indicates the shorter square while
the configuration \verb+[  )(  )+ indicates the longer square.
The solid underline indicates $u_2$ while the dotted underline indicates
$\o{u}_2$. The extension of $U^2$ is the final suffix not
in bold.
The \fsds\ $\cal U$ can be non-trivially cyclically shifted to the right by 
at most \lcp{u} = $lcp(u_2\o{u}_2,\o{u}_2u_2)=
lcp(aaabaa,aaaaab)=3$ positions, thus every subfamily has at most 4 {\fsds}s.
Note, however, that the \invf\ $aaaaabaaabaa$ -- highlighted in Figure~4 -- cyclically
shifts within a subfamily and then returns to the 
original position for the first \fsds\ of each segment. There is 1 \aseg\ and 2 \bseg{s}
since $(\p{U}-\q{U})/2=2$, $t$ can take the $3$ values 0, 1, or 2. For each new
segment, the size of the shorter square decreases by a multiple of $|u_1|$ while 
the size of the longer square remains constant.

\begin{figure}
\leftskip=-18pt
\includegraphics[scale=0.72]{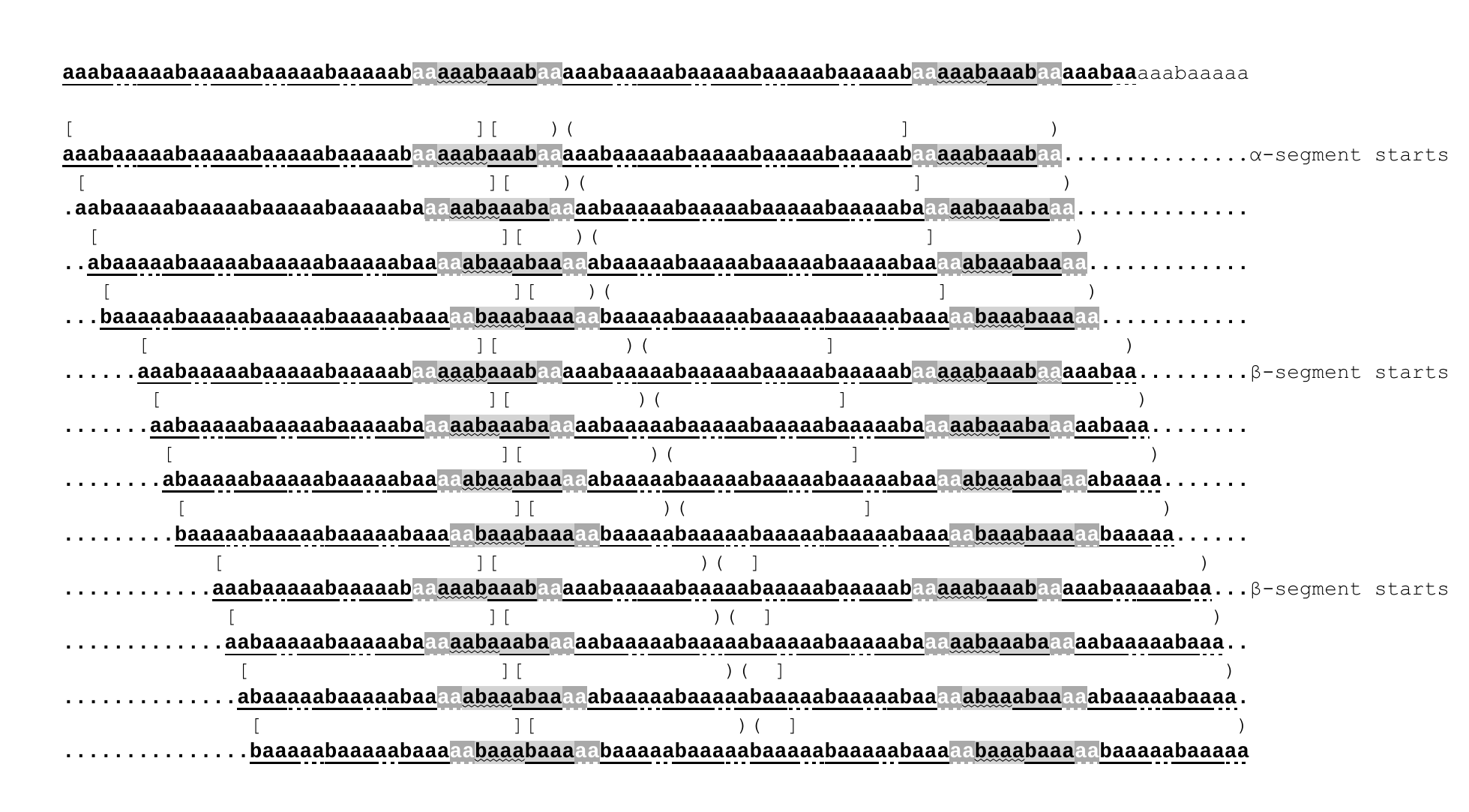}
\caption{Example of an \bfam\ of $\cal U$}
\end{figure}

\begin{claim}
\label{bottomless-beta}
Let $x$ be a string starting with an \bfam\ of an \fsds\ $\cal U$ and let $\cal V$
be the last member of the $\cal U$ family. Let every \fsds\ $\cal W$ after $\cal V$
be so that $R_1({\cal U})\leq \sbig{{\cal W}}\leq\send{u_{[1]}}$. 
Then $\delta(x)\leq \frac{5}{6}{|x|}\-\frac{1}{3}{|u|}$.
\end{claim}

\begin{proof}
\onehalfspacing
Let the type of $\cal V$ be $(\p{U}\-t,\q{U}\+t)$. Then $2t\leq \p{U}\-\q{U}$.
Since every \fsds\ $\cal W$ after $\cal V$ starts after $R_1$ but ends
before $\send{u_{[1]}}$, the total number of \fsds{s} in $x$ is the number of \fsds{s} in 
the $\cal U$ family plus possibly $\leq |u_1|$ additional \fsds{s}, i.e.
$f\leq (t\+2)|u_1|$.
Since $|x|\geq |U^2|\+f=2(\p{U}\+\q{U})|u_1|\+2|u_2|\+f$, we get\\
$\frac{5}{6}{|x|}\-\frac{1}{3}{|u|}\geq \frac{5}{6}2(\p{U}\+\q{U})|u_1|\+\frac{5}{6}2|u_2|\-
\frac{1}{3}\p{U}|u_1|\-\frac{1}{3}|u_2|=
\frac{4\p{U}\+5\q{U}}{3}|u_1|\+$\linebreak
$\frac{4}{3}|u_2|> \frac{4\p{U}\-4\q{U}}{3}|u_1|\+\frac{9\q{U}}{3}|u_1|>
\frac{8t}{3}|u_1|\+2|u_1|>$
$t|u_1|\+2|u_1|\geq f=\delta(x)$.
\end{proof}

\begin{claim}
\label{bottom-beta}
Let $x$ be a string starting with an \bfam\ of an \fsds\ 
$\cal U$ and let there be
some \fsds{s} in $x$ that are not members of the $\cal U$ family.
Let for any $\cal V$ that is not a member of the $\cal U$ family, 
$\delta({x'})\leq \frac{5}{6}{|x'|}\-\frac{1}{3}{|v|}$ where $x'$ is a suffix of $x$
starting at the same position as $\cal V$.
Then $\delta({x})\leq \frac{5}{6}{|x|}\-\frac{1}{3}{|u|}$.
\end{claim}

\begin{proof}
\onehalfspacing
Let the last \bseg\ be of type $(\p{U}\-t,\q{U}\+t)$. Then $\p{U}\-t\geq \q{U}\+t$
and so $2t\leq \p{U}\-\q{U}$ and the size of the $\cal U$ family is $\leq (t\+1)|u_1|$.
By Lemma~\ref{vv-cases}, $\cal V$ is either a \dmate, or a \gmate, or a \emate\
of $\cal U$. Since $\cal U$ family is an \bfam, $\cal V$ cannot be \gmate\ of $\cal U$.
The size of the $\cal U$ family is $f\leq (t\+1)|u_1|$.\\
Let us first discuss the case when $\cal V$ is a \dmate\ of $\cal U$.
Then $T({\cal U},{\cal V})\geq f$, $T({\cal U},{\cal V})\geq (\p{U}\+\q{U}\-1)|u_1|\+|u_2|$
and so $\frac{1}{2}|G|\+\frac{1}{3}|T|> \frac{1}{2}f\+\frac{\p{U}\+\q{U}\-1}{3}|u_1|>
\frac{1}{2}f\+\frac{\p{U}\-\q{U}}{3}|u_1|\+$\linebreak
$\frac{2\q{U}\-1}{3}|u_1|\geq
\frac{1}{2}f\+\frac{2t}{3}|u_1|\+\frac{1}{3}|u_1|>
\frac{1}{2}f\+\frac{2t\+1}{3}|u_1|>\frac{1}{2}f\+\frac{t\+1}{2}|u_1|\geq
\frac{1}{2}f\+\frac{1}{2}f=f$. Thus, by Lemma~\ref{GapTail},
$\delta(x)\leq \frac{1}{2}|x|\-\frac{1}{3}|u|$.\\
Let assume that $\cal V$ is an \emate\ of $\cal U$.\\
If there were no \eemate\ of $\cal U$, then by Lemma~\ref{bottomless-beta},
$\delta(x)\leq \frac{5}{6}|x|\-\frac{1}{3}|u|$. 
So let us assume that there is a \eemate, and let $\cal V$ be the first \eemate\ of $\cal U$.
Between the first \emate\ of $\cal U$ and $\cal V$ there are at most $|u_1|$ \fsds{s},
$\delta(x)\leq \delta(x')\+(t\+2)|u_1|$.
By the assumption of this lemma, $\delta(x')\leq \frac{1}{2}|x'|\-\frac{1}{3}|v|$.
By Lemma~\ref{supere}, there are two cases:
\begin{my_itemize}
\leftskip=-10pt
\item[$(a)$] $G({\cal U},{\cal V})\geq (2\p{U}\+\q{U}\-3)|u_1|\+2|u_2|$ and
$T({\cal U},{\cal V})\geq (\p{U}\+\q{U}\-3)|u_1|\+|u_2|$.

\smallskip
Since $\q{U}\geq 1$ and $t\geq 2$,
then $\frac{1}{2}|G|\+\frac{1}{3}|T|>
\frac{2\p{U}\+\q{U}\-3}{2}|u_1|\+$\linebreak
\smallskip
$\frac{\p{U}\+\q{U}\-2}{3}|u_1|=\frac{8\p{U}\+5\q{U}\-13}{6}|u_1|=
\frac{8\p{U}\-8\q{U}}{6}|u_1|\+\frac{13\q{U}\-13}{6}|u_1|>$\linebreak
\smallskip
$\frac{16t}{6}|u_1|=
t|u_1|\+\frac{10t}{6}|u_1|\geq
t|u_1|\+\frac{20}{6}|u_1|\geq t|u_1|\+2|u_1|$ as $t\geq 2$.
\smallskip
\item[$(b)$] $G({\cal U},{\cal V})\geq \p{U}|u_1|\+|u_2|$ and
$T({\cal U},{\cal V})\geq (\p{U}\+\q{U}\-1)|u_1|\+|u_2|$.

\smallskip
Then $\frac{1}{2}|G|\+\frac{1}{3}|T|>\frac{\p{U}}{2}|u_1|\+\frac{\p{U}\+\q{U}\-1}{3}|u_1|=
\frac{5\p{U}\+2\q{U}\-2}{6}|u_1|=$\linebreak
\smallskip
$\frac{5\p{U}\-5\q{U}}{6}|u_1|\+\frac{7\q{U}\-2}{6}|u_1|\geq \frac{10t}{6}|u_1|\+\frac{5}{6}|u_1|=t|u_1|\+\frac{4t}{6}|u_1|\+\frac{5}{6}|u_1|\geq
t|u_1|\+$\linebreak
\smallskip
$\frac{8}{6}|u_1|\+\frac{5}{6}|u_1|=
t|u_1|\+\frac{13}{6}|u_1|>t|u_1|\+2|u_1|$ as $t\geq 2$.
\end{my_itemize}
\end{proof}

\vspace{-25pt}

\subsubsection{The case $\cal U$-family consists of all three \amate{s}, \bmate{s}, and \gmate{s}}
\label{family-alpha-beta-gamma}

\smallskip
\noindent
We must first estimate the size of the family. We proceed by
investigating its structure. Since there must be some \bmate{s}
of $\cal U$, $\p{U}\geq \q{U}\+2$. The family consists of segments.

The first segment consists of $\cal U$ and possibly its right cyclic shifts,
i.e. its \amate{s}. The size of the segment is $\leq$ \lcp{u} $\leq |u_1|\-2$, see Lemma~\ref{rot2}. All the \fsds{s} in this segments have 
the first exponent equal to $\p{U}$ and the second exponent
equal to $\q{U}$, thus we say that the type of the segment is $(\p{U},\q{U})$.

Then there must be a \bmate\ of $\cal U$ and possibly its right cyclic shifts.
 All the \fsds{s} in the segment have the first exponent equal to 
$\p{U}\-i_1$ and the second exponent equal to $\q{U}\+i_1$ for some 
$1\leq i_1<(\p{U}\-\q{U})/2$, thus we say that the type of the segment is
$(\p{U}\-i_1,\q{U}\+i_1)$. This so-called \bseg\ has size $\leq lsp(u_1,\c{u}_1)\leq |u_1|\-2$ if $\p{U}\-i_1>\q{U}\+i_1$, see Lemma~\ref{rot2},
or $\leq |u_2|\-1\leq |u_1|\-2$ if $\p{U}\-i_1=\q{U}\+i_1$. Hence the
\bseg\ has size $\leq |u_1|\-2$.

Then there may be another \bseg\
of type $(\p{U}\-i_2,\q{U}\+i_2)$ for some $1\leq i_1<i_2<(\p{U}\-\q{U})/2$, etc. There may be $t$ such \bseg{s} where $2t\leq \p{U}\-\q{U}$.
Let the last \bseg\ have type $(p,q)$; then $p\geq q$.

\begin{figure}
\leftskip -10pt
\includegraphics[scale=0.81]{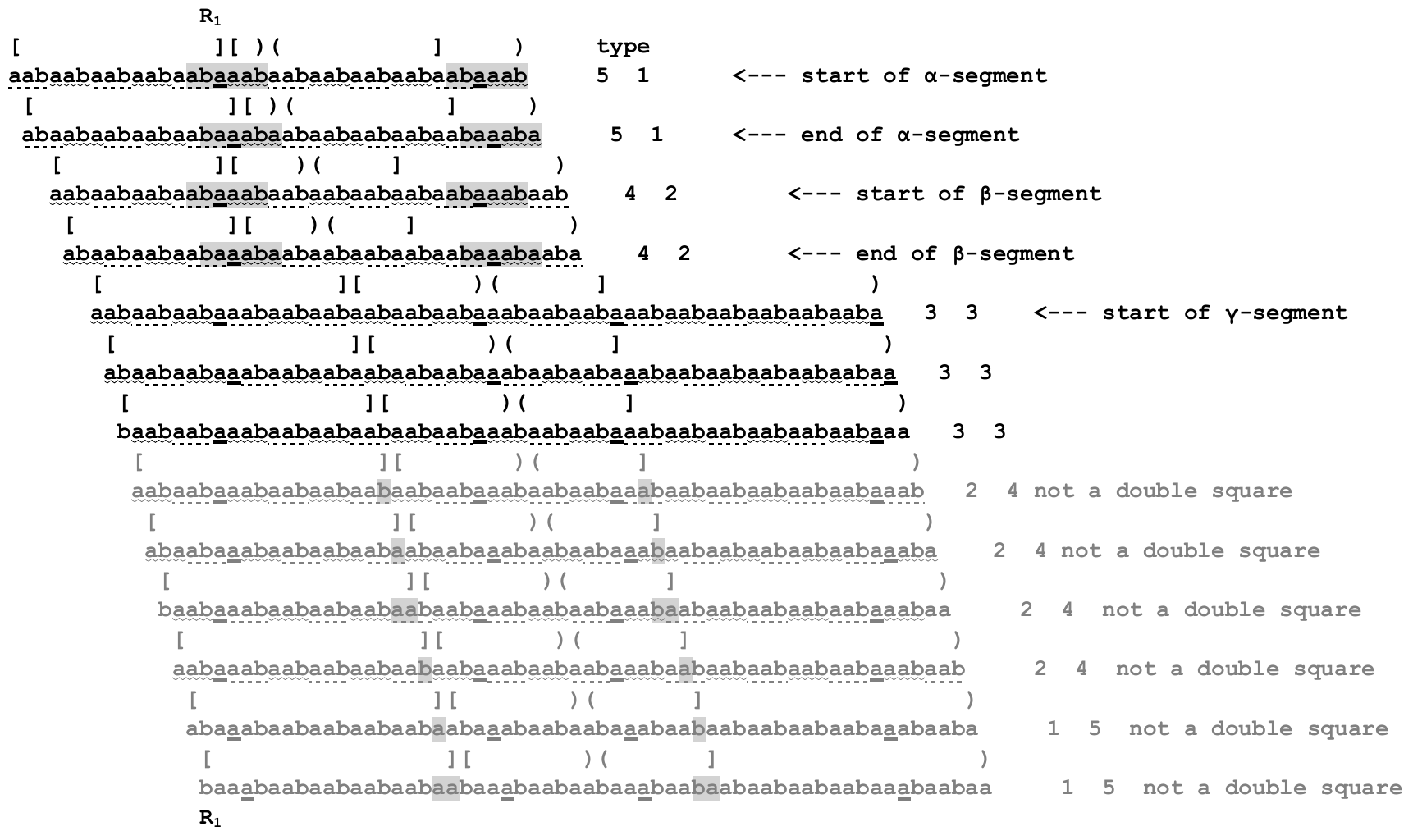}
\caption{Example of a \gfam\ of $\cal U$}
\end{figure}

Then there must be $\cal G$, a \gmate\ of $\cal U$.
Consider all the \gmate{s} of $\cal U$ of which $\cal G$ is the first one.
They form what we call a \gseg. Since all the \fsds{s} in the \gseg\
have the short square of the same length $|U^2|$ and since
they have equal exponents by Lemma~\ref{gamma},
by Lemma~\ref{vv-cases} they are all \amate{s} of $\cal G$.
Thus,  the \gseg\ consists of a \gmate\ of 
$\cal U$ and its right cyclic shifts. The  shorter square of $\cal G$ has 
a form\newline
$[s_1{u_1}^i{u_2}{u_1}^{(\p{U}\+\q{U}\-i\-1)}s_2][s_1{u_1}^i{u_2}{u_1}^{(\p{U}\+\q{U}\-i\-1)}s_2]$ 
for some $1\leq i\leq p$ and some
$s_1$ and $s_2$ such that $s_2s_1=u_1$. In order to estimate the
size of the \gseg, we have to estimate how many right cyclic shifts 
$\cal G$ can have. First we need to discuss the difference between a 
\ds\ structure and an \fsds: it is quite possible to have a \ds\ structure
in a string that is not an \fsds\ as there is a farther occurrence of the shorter or the longer square of the \ds\ structure. Thus, we always overestimate the sizes of $\cal U$
families, as we really count the \ds\ structures and up to $|u_1|$ cyclic shifts for each \aseg\ or \bseg.
We know that actually every segment can have
at most $lcs(u_1,\c{u}_1)\+lcp(u_1,\c{u}_1)\leq |u_1|\-2$ members.
So, we can imagine every segment to have a ``hole".
So if there is a farther \fds\ that can be assigned to the hole, we will say that it \emph{complements}
the segment and thus does not need to be counted as its count
was already part of the overestimation.
If there is a farther \fds\ $\cal V$ containing a farther copy of ${u_1}^ru_2{u_1}^ru_2$ and 
thus implying that though there is a structure of a \ds\
of type $(r,r')$, it is not an \fsds, we will say that $\cal V$ \emph{replaces} the \ds\ structure of type $(r,r')$.

Now back to estimating the size $f$ of an \gfam. We shall show that $f\leq \frac{2}{3}{(\p{U}\+1)}|u_1|$. There are basically two cases:
\begin{my_itemize}
\leftskip=-10pt
\item[$(i)$] $\cal G$, the first member of the \gseg, is of type $(\p{U}\-t,\q{U}\+t)$ and $\p{U}\-t> 2(\q{U}\+t)$.\\
Since $\p{U}\-t> 2(\q{U}\+t)$, $3t<\p{U}\-2\q{U}$ and so
$3t\leq \p{U}\-2\q{U}\-1$ and thus
$6t\leq 2\p{U}\-4\q{U}\-2$.
By Lemma~\ref{gamma} and Lemma~\ref{rot1}, $\cal G$ has $\leq (\q{U}\+t)\-1$
cyclic shifts.  Thus, we start with $\cal U$ of type $(\p{U},\q{U})$ and end
with the last member of the \gseg\ that is of type\newline
$(\p{U}\-t\-(\q{U}\+t\-1)),(\q{U}\+t\+(\q{U}\+t\-1))$, thus there are at most\linebreak
$(2\q{U}\+2t\-1)-\q{U}\+1=\q{U}\+2t$ members in the \gfam.
Then $3f=3\q{U}\+6t\leq 3\q{U}\+2\p{U}\-4\q{U}\-2=
2\p{U}\-\q{U}\-2\leq 2\p{U}\-3<2\p{U}\+2=2(\p{U}\+1)$ as $q\geq 1$.
Thus, $f<\frac{2}{3}(\p{U}\+1)|u_1|$.
\item[$(ii)$] ${\cal G}$, the first member of the \gseg, is of type $(\p{U}\-t,\q{U}\+t)$ and $\p{U}\-t\leq 2(\q{U}\+t)$.
\begin{my_itemize}
\leftskip=-25pt
\item[$(ii_1)$] $\p{U}\-t\leq \q{U}\+t$\\
By Lemma~\ref{gamma}, $G^2$ of $\cal G$ contains a further copy of\newline ${u_1}^{\q{U}\+t}u_2{u_1}^{\q{U}\+t}u_2$ and so $\cal G$ either ``replaces" a possible member of the \aseg\ or a \bseg, or it ``complements" the \aseg\ or a \bseg.
Thus, $f\leq \frac{1}{2}{(\p{U}\-\q{U})}|u_1|<\frac{2}{3}{(\p{U}\+1)}|u_1|$.
\item[$(ii_2)$] $\p{U}\-t > \q{U}\+t$.\\
Either $g_2$ of $\cal G$ is small, i.e. $|g_2|<|u_1|$ and then $\cal G$ has less than $|u_1|$ shifts, and so
$f\leq \frac{1}{2}{(\p{U}\-\q{U})}|u_1|\+|u_1|\leq \frac{2}{3}{(\p{U}\+1)}|u_1|$,
or $|g_2|\geq |u_1|$.\\
Thus assume that  $|g_2|\geq |u_1|$. We can further assume by Lemma~\ref{gamma}
that the last member of the \gseg\ is of type $(\q{U}\+t,\p{U}\-t)$, since if it were shifted
any further, it would start ``replacing" or ``completing" the members of the 
\aseg\ or the \bseg{s}, so we do not need to count them.

Since $\p{U}\-t\leq 2(\q{U}\+t)$, then $\p{U}\-2\q{U}\geq 3t$.
Thus $3f=$\newline
$3(\p{U}\-t\-\q{U}\-1)|u_1|=(3\p{U}\-3t\-3\q{U}\+3)|u_1|\leq$\newline
$(3\p{U}\-3\q{U}\+3\+2\q{U}\-\p{U})|u_1|=
(2\p{U}\-\q{U}\+3)|u_1|\leq$\newline
$(2\p{U}\+2)|u_1|=
2(\p{U}\+1)|u_1|$. Therefore, $f\leq \frac{2}{3}{(\p{U}\+1)}|u_1|$.
\end{my_itemize}
\end{my_itemize}

\begin{claim}
\label{bottomless-gamma}
Let a string $x$ start with an \gfam\ of an \fsds\ $\cal U$ and let there be no other
\fsds{s}. Then $\delta(x)\leq \frac{5}{6}|x|\-\frac{1}{3}|u|$.
\end{claim}

\begin{proof}
\onehalfspacing
The size of the family $f\leq \frac{2}{3}(\p{U}\+1)|u_1|$ and so
$\frac{1}{6}f\leq \frac{2}{18}(\p{U}\+1)|u_1|$. 
$|x|\geq f\+|U^2|=f\+2(\p{U}\+\q{U})|u_1|\+2|u_2|$, and so
$\frac{5}{6}|x|\-\frac{1}{3}|u|\geq \frac{5}{6}f\+$\\
$\frac{5}{6}(2(\p{U}\+\q{U})|u_1|\+
\frac{5}{6}2|u_2|\-\frac{1}{3}\p{U}|u_1|\-\frac{1}{3}|u_2|=
\frac{5}{6}f\+\frac{8}{6}p|u_1|\+\frac{10}{6}\q{U}|u_1|\+\frac{3}{6}|u_2|>
\frac{5}{6}f\+\frac{30}{18}p|u_1|>
\frac{5}{6}f\+\frac{2}{18}p|u_1|\+\frac{28}{18}p|u_1|\geq
\frac{5}{6}f\+\frac{2}{18}(p\+1)|u_1|\geq\frac{5}{6}f\+\frac{1}{6}f= f =\delta(x)$.
\end{proof}

\begin{claim}
\label{bottom-gamma}
Let a string $x$ start with an \gfam\ of an \fsds\
$\cal U$. Let $\cal V$ be the first \fsds\ not
in the $\cal U$ family. Let $x'$ be the suffix of $x$ starting at the same position
as $\cal V$. Let $\delta(x')\leq \frac{5}{6}|x'|\-\frac{1}{3}|v|$. Then
$\delta(x)\leq \frac{5}{6}|x|\-\frac{1}{3}|u|$.
\end{claim}

\begin{proof}
$\cal V$ can be either a \dmate\ or \emate\ of $\cal U$. 
Let $\cal G$ be the last member of the \gseg\ and let its type be $(\p{U}\-t,\q{U}\+t)$.
Then $g^2$ has the format\\
${u_1}^ts_1[s_2{u_1}^{(\p{U}\-t\-1)}u_2{u_1}^\q{U}s_1][s_2{u_1}^{(\p{U}\-t\-1)}u_2{u_1}^\q{U}s_1]$. If $\send{v_{[1]}}\leq \send{g_{[1]}}$, then by Lemma~\ref{vv-cases} $\cal U$ would be a \bmate\ of $\cal G$, which is impossible as by Lemma~\ref{gamma}, $\p{G}=\q{G}$. Thus $\send{v_{[1]}}>\send{g_{[1]}}$.

\vspace{-10pt}
\begin{my_itemize}
\onehalfspacing
\leftskip=-7pt
\item[$(a)$] Let $\cal V$ be a \dmate.\\
Then we are assured that $T({\cal U},{\cal V})\geq (\p{U}\+\q{U}\-1)|u_1|$.
But a little bit more is true. Clearly,
$v_{[1]}$ contains an \invf\ from ${\big[}L_1({\cal U}),R_1({\cal U}){\big]}$.
If $\sbig{v_{[2]}}\leq R_2({\cal U})$, then $v_{[2]}$ would contain an \invf\
from ${\big[}L_2({\cal U}),R_2({\cal U}){\big]}$, giving $|v|=|w|$, a contradiction. Hence
$\sbig{v_{[2]}} > R_2({\cal U})$ and by Lemma~\ref{syncpr}, $T({\cal U},{\cal V})\geq (\p{U}\+\q{U})|u_1|$.\\
Since $G({\cal U},{\cal V})\geq f$, we have
$\frac{1}{2}|G|\+\frac{1}{3}|T|\geq \frac{1}{2}f\+
\frac{1}{3}(\p{U}\+\q{U})|u_1|\geq$\newline
$\frac{1}{2}f\+\frac{1}{3}(\p{U}\+1)|u_1|\geq \frac{1}{2}f\+\frac{1}{2}f=f$ as 
$\q{U}\geq 1$ and  $\frac{1}{2}f\leq \frac{1}{3}{(\p{U}\+1)}|u_1|$.
\item[$(b)$] Let $\cal V$ be an \emate\ of $\cal U$, but not a \eemate.\\
So $\sbig{v_{[1]}}\leq \send{u_{[1]}}$ and $\send{v_{[1]}}>\send{g_{[1]}}$.
By Lemma~\ref{syncpr}, $T({\cal U},{\cal V})\geq (\p{U}\+\q{U})|u_1|$ and so
$\frac{1}{2}|G|\+\frac{1}{3}|T|\geq \frac{1}{2}f\+$\newline
$\frac{1}{3}(\p{U}\+1)|u_1|\geq
\frac{1}{2}f\+\frac{1}{2}f=f$.
\item[$(c)$] Let $\cal V$ be a \eemate\ of $\cal U$.\\
By Lemma~\ref{supere}, there are two possibilities:
\begin{my_itemize}
\leftskip=-20pt
\item[$(c_1)$]  $G\geq (2\p{U}\+\q{U}\-3)|u_1|$ and $T\geq (\p{U}\+\q{U}\-2)|u_1|$\\
Then $\frac{1}{2}|G|\+\frac{1}{3}|T|\geq
\frac{6\p{U}\+3\q{U}\-9\+2\p{U}\+2\q{U}\-4}{6}|u_1|=$\newline
\smallskip
$\frac{8\p{U}\_5\q{U}\-13}{6}|u_1|=
\frac{4\p{U}\+4\p{U}\+5\q{U}\-13}{6}|u_1|$. Since $\p{U}\geq 4$ and\linebreak
$\q{U}\geq 1$,
$\frac{1}{2}|G|\+\frac{1}{3}|T|\geq\frac{4\p{U}\+16\+5\-13}{6}|u_1|=
\frac{4\p{U}\+8}{6}|u_1|>\frac{4\p{U}\+4}{6}|u_1|=f$.
\item[$(c_2)$] $G\geq \p{U}|u_1|$ and $T\geq (\p{U}\+\q{U}\-1)|u_1|$\\
\smallskip
$\frac{1}{2}|G|\+\frac{1}{3}|T|\geq \frac{3\p{U}\+2\p{U}\+2\q{U}\-2}{6}|u_1|=
\frac{2\p{U}\+3\p{U}\+2\q{U}\-2}{6}|u_1|\geq$\newline
\smallskip
$\frac{2\p{U}\+16\+2\-2}{6}|u_1|=\frac{2\p{U}\+12}{6}|u_1|>
\frac{2\p{U}\+2}{16}|u_1|\geq f$, since $\p{U}\geq 4$ and
$\q{U}\geq 1$.
\end{my_itemize}
\end{my_itemize}
\end{proof}

\vspace{-30pt}

\subsection{New upper bounds}\label{ubound}

\smallskip
\begin{theorem}
The number of \fsds{s} in a string of length $n$ is bounded by $\lfloor 5n/6 \rfloor$.
\end{theorem}

\begin{proof}
We prove by induction the following, a slightly stronger, statement: $\delta(x)\leq \frac{5}{6}|x|\-\frac{1}{3}|u|$ for $|x|\geq 10$ where 
$u$ is the generator of the shorter square of the first \fsds\ of $x$.
We do not have to consider strings of length 9 or less, as such strings
do not contain \fsds{s}. Since a string of length 10 contains at most one
\fsds\ (see the note after Definition~\ref{notat}), the statement is true for strings of size 10. Assuming
the statement is true for all $|x|\leq n$, we shall prove it holds for all $|x|\leq n+1$.\\

\noindent
If  $x=x[1..n\+1]$ does not start with an \fsds, then $\delta(x)=\delta(x[2..n\+1])\leq \frac{5}{6}|x[2..n\+1]|\-
\frac{1}{3}|u|\leq
 \frac{5}{6}|x[1..n\+1]|\-\frac{1}{3}|u|$. Thus, we can assume that $x$ starts with an {\fsds} $\cal U$.
If $\cal U$ is the only \fsds\ of $x$, then $|x|\geq 2|u|$, thus the statement is obviously true. Therefore, we can assume that 
$x$ starts with a {\fsds} $\cal U$ and $\delta(x)\geq 2$.\\\\
Case $(a)$ assume that $x$ starts with an \afam\ of $\cal U$.\\
If there is no further \fsds\ in $x$, by Claim~\ref{bottomless-alpha}, the assertion is true.
Otherwise, we carry out the induction step by Claim~\ref{bottom-alpha}.\\
Case $(b)$ assume that $x$ starts with an \bfam\ of $\cal U$.\\
If there is no further \fsds\ in $x$, by Claim~\ref{bottomless-beta}, the assertion is
true. Otherwise, we carry out the induction step by Claim~\ref{bottom-beta}.\\
Case $(c)$ assume that $x$ starts with an \gfam\ of $\cal U$.\\
If there is no further \fsds\ in $x$, by Claim~\ref{bottomless-gamma}, the assertion is
true. Otherwise, we carry out the induction step by Claim~\ref{bottom-gamma}.
\end{proof}

\begin{corollary}
The number of distinct squares in a string of length $n$ is bounded by $\lfloor 11n/6 \rfloor$.
\end{corollary}

\begin{proof}
The number of distinct squares in a string is the sum of the number of  \fsds{s} plus the number of single rightmost squares.
Since, for a string of length $n$,  the number of \fsds{s} is bounded by $\lfloor 5n/6 \rfloor$, the number of distinct squares is bounded 
by $\lfloor (2\cdot 5/6\+1/6)n \rfloor$; that is, by  $\lfloor 11n/6 \rfloor$.
\end{proof}

\section{Proofs}
\label{hard-proofs}

\subsection{Proof of Lemma~\ref{invfactor}:}
\label{proof-invfactor}

\smallskip
\noindent
Assume, in order to derive a contradiction, that an \invf\ $\inv{v}$ 
occurs to the left of $L_1$.
Consider the \invf\ $\inv{w}$ starting at the position $L_1$. Then $w_1$, $w_2$ and $\o{w}_2$ are 
left cyclic shifts of, respective $u_1$,  $u_2$ and $\o{u_2}$. Since $\inv{w}$ cannot be further cyclically shifted to 
the left, $lcs(w_2\o{w}_2,\o{w}_2w_2)=0$. Since $\inv{v}$ is occurring to the left of $\inv{w}$,
there are non-empty strings $a$ and $c$ and a string $b$ so that
$|a|>|b|$ and $a\inv{w}=b\inv{v}c$.
We split the argument into several cases depending on where the
\invf\ $\inv{v}$ ends.

\begin{my_enumerate}
\leftskip=-7pt
\item Case when $\inv{v}$ ends in the second copy of $\o{w}_2$ in the 
\invf\ $\inv{w}$:

\includegraphics[scale=0.7]{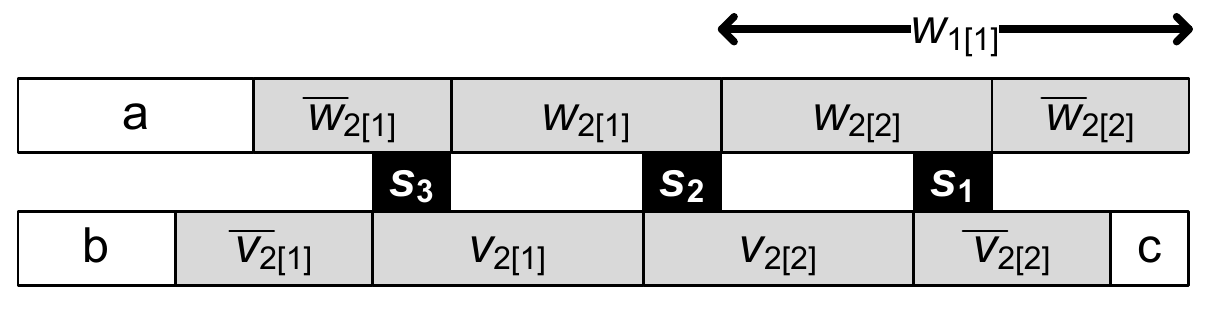}

\noindent Let $s_1$ be the overlap of $w_{2[2]}$ and 
$\o{v}_{2[2]}$.
Then $s_1$ is  a non-trivial proper prefix of $\o{v}_2$ and a non-trivial proper suffix of $w_2$.
There is a copy $s_2$ of $s_1$ as a suffix of $w_{2[1]}$, 
and it must be a
prefix of $v_{2[2]}$ as $|w_2|=|v_2|$. Consequently,
there is a copy $s_3$ of $s_2$ as a prefix of $v_{2[1]}$, and
it must be a suffix of $\o{w}_{2[1]}$ as $|s_3|=|s_1|\leq |\o{w}_2|$ and
$|v_{2[1]}|\+|v_{2[2]}|\+|\o{v}_{2[2]}|=$
$|w_{2[1]}|\+|w_{2[2]}|\+|\o{w}_{2[2]}|$.
Thus, $s_1$ is a suffix of both $w_2$ and of $\o{w}_2$, 
 contradicting the fact that $lcs(w_2\o{w}_2,\o{w}_2w_2)=0$.

\item  Case when $\inv{v}$ ends in the second copy of $w_2$ of $\inv{w}$:

\includegraphics[scale=0.7]{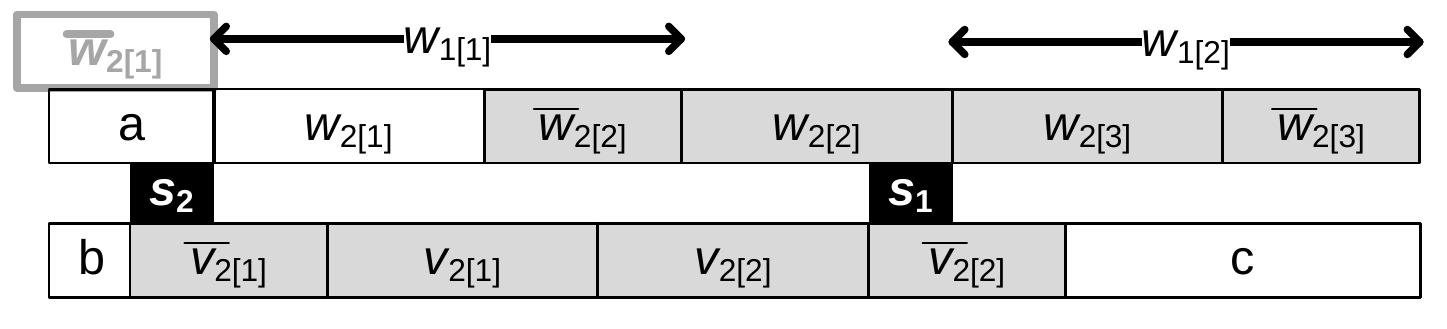}

\noindent  Let $s_1$ be the overlap of $w_{2[2]}$ and 
$\o{v}_{2[2]}$. Then $s_1$ is  suffix of $w_2$ and a prefix of $\o{v}_2$. There is a copy $s_2$ of $s_1$ as a prefix of 
$\o{v}_{2[1]}$. Consequently, $s_2$ must be a suffix
of $\o{w}_{2[1]}$ since $|\o{v}_{2[1]}|\+|v_{2[1]}|\+|v_{2[2]}|=|w_{2[1]}|\+|\o{w}_{2[2]}|\+|w_{2[2]}|$.
Thus, $s_1$ is a suffix of both $w_2$ and $\o{w}_2$, 
contradicting the fact that $lcs(w_2\o{w}_2,\o{w}_2w_2)=0$.\\
\textit{Note that the whole of $\o{w}_{2[1]}$ might not be a part of the string (and that is why in the diagram it is depicted in gray), in which case $a$ is a non-trivial proper suffix of $\o{w}_{2[1]}$, and the argument holds.}

\item Case when $\inv{v}$ ends in the first copy of $w_2$ of $\inv{w}$:

\includegraphics[scale=0.7]{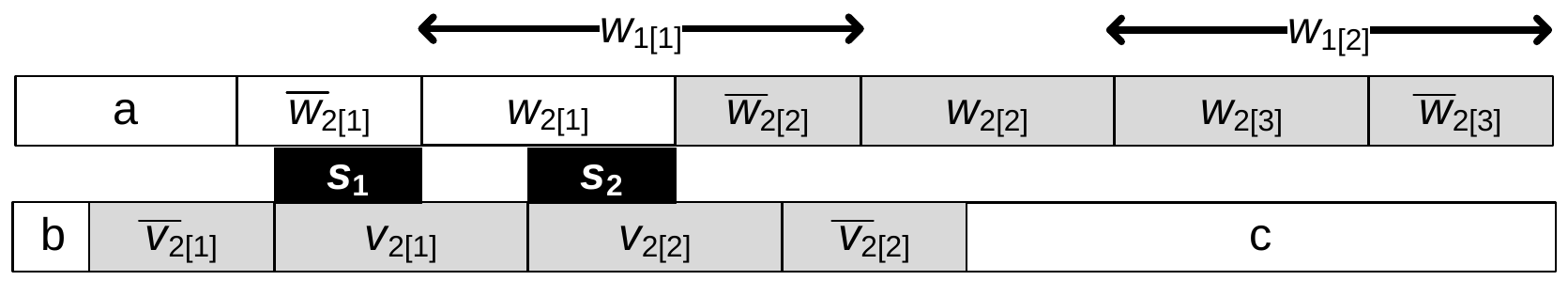}

\noindent Let $s_1$ be the overlap of $\o{w}_{2[1]}$ and $v_{2[1]}$. Then $s_1$ is a suffix of
$\o{w}_2$ and a prefix of $v_2$. There is a copy $s_2$ of $s_1$ as a prefix of $v_{2[2]}$.
It must be a suffix of $w_{2[1]}$ since $|\o{v}_{2[1]}|\+|v_{2[1]}|=|\o{w}_{2[1]}|\+|w_{2[1]}|$.
Thus, $s_1$ is a suffix of both $w_2$ and $\o{w}_2$, contradicting the fact that $lcs(w_2\o{w}_2,\o{w}_2w_2)=0$.

\item Case when $\inv{v}$ ends in the first copy of $\o{w}_2$ of $w_2\o{w}_2,\o{w}_2w_2$, or lies completely outside of $w_2w_2\o{w}_2$ of $\inv{w}$ and ends in $\o{w}_2$:

\includegraphics[scale=0.7]{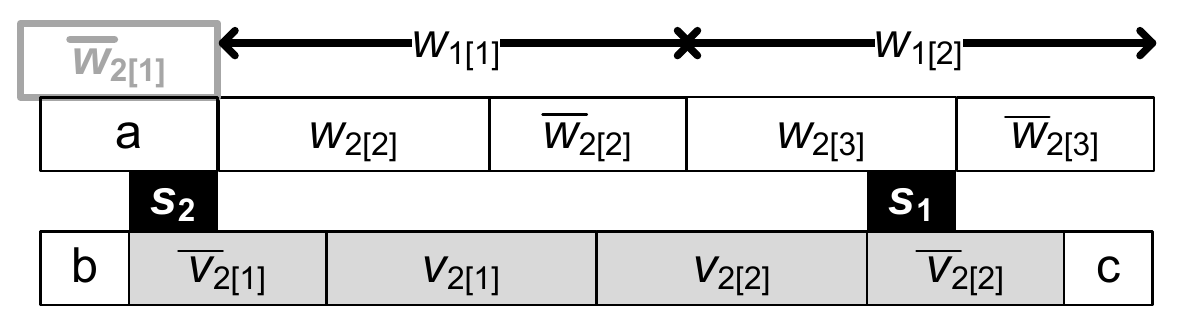}

\noindent Let $s_1$ be the overlap of $w_{2[3]}$ and $\o{v}_{2[2]}$. Then $s_1$ is a suffix of $w_2$ and a prefix of $\o{v}_2$.
There is a copy $s_2$ of $s_1$ as a prefix of $\o{v}_{2[1]}$.
It must be a suffix of $\o{w}_{2[1]}$ as
$|w_{2[2]}|\+|\o{w}_{2[2]}|\+|w_{2[3]}|=
|\o{v}_{2[1]}|\+|v_{2[1]}|\+|v_{2[2]}|$. Thus, $s_1$ is a suffix of both $w_2$ and $\o{w}_2$, contradicting the fact that
$lcs(w_2\o{w}_2,\o{w}_2w_2)=0$.\\
\textit{Note that the whole of $\o{w}_{2[1]}$ might not be a part of the string (and that is why in the diagram it is depicted in gray), in which case $a$ is a non-trivial proper suffix of $\o{w}_{2[1]}$, and the argument holds.}

\item Case when $\inv{v}$ lies completely outside of $\inv{w}$ and ends in $w_2$:

\includegraphics[scale=0.7]{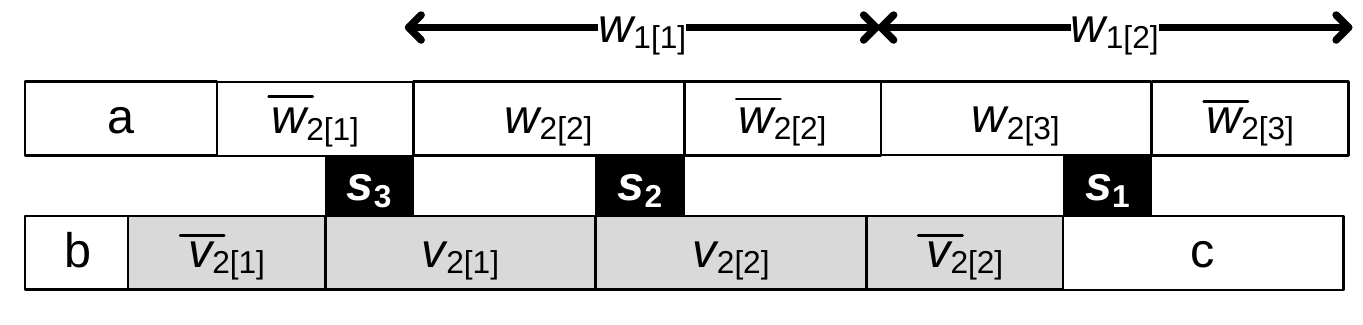}

\noindent Let $s_1$ be the offset of $a\inv{w}$ and $b\inv{v}$,
i.e. $a\inv{w}=b\inv{v}s_1$. Then $s_1$ is a suffix of $w_2$.
There is a copy $s_2$ of $s_1$ as a suffix of $\o{w}_{2[2]}$.
It must be a prefix of $\o{v}_{2[2]}$ as
$|\o{w}_{2[2]}|\+|w_{2[3]}|=|v_{2[2]}|\+|\o{v}_{2[2]}|$. 
There is a copy $s_3$ of $s_2$ as a prefix of $v_{2[1]}$. It must be
a suffix of $\o{w}_{2[1]}$ as
$|w_{2[2]}|\+|\o{w}_{2[2]}|\+|w_{2[3]}|=|v_{2[1]}|\+|v_{2[2]}|\+|\o{v}_{2[2]}|$.
Thus, $s_1$ is a suffix of both $w_2$ and $\o{w}_2$, contradicting the fact that $lcs(w_2\o{w}_2,\o{w}_2w_2)=0$.
\end{my_enumerate}

As a second step of the proof, let us investigate whether an \invf\ 
$\inv{v}$ can occur to the right of $R_1$ while ending
before $L_2$. The proof of this step is essentially the same argumentation
as for the first one, so though added for the sake of completion, it is
presented in an abbreviated form, i.e. we just present the diagrams and 
the conclusions.

Consider the \invf\ $\inv{w}$ starting at the position $R_1$. Then $w_1$ respective $w_2, \o{w}_2$ are  
right cyclic shifts of $u_1$ respective $u_2, \o{u}_2$. Moreover,
$lcp(w_2\o{w}_2,\o{w}_2w_2)=0$ as $\inv{w}$ cannot be shifted
right. Since $\inv{v}$ is occurring to the right of $\inv{w}$,
there are non-empty strings $b$ and $c$ and a string $a$ so that
$|a|<|b|$ and $a\inv{w}c=b\inv{v}$.
We split the argument into several cases depending on where the
\invf\ $\inv{v}$ starts.

\begin{my_enumerate}
\leftskip=-7pt
\item Case when $\inv{v}$ starts in the first copy of $\o{w}_2$ in the 
\invf\ $\inv{w}$:

\includegraphics[scale=0.7]{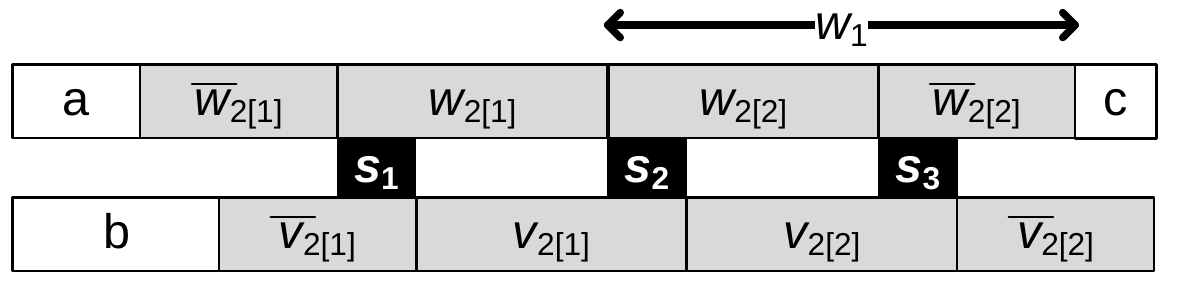}

\noindent Then $s_1$ is both a prefix of $w_2$ and $\o{w}_2$, 
 contradicting the fact that $lcp(w_2\o{w}_2,\o{w}_2w_2)=0$.

\item  Case when $\inv{v}$ starts in the first copy of $w_2$ in the 
\invf\ $\inv{w}$:

\includegraphics[scale=0.7]{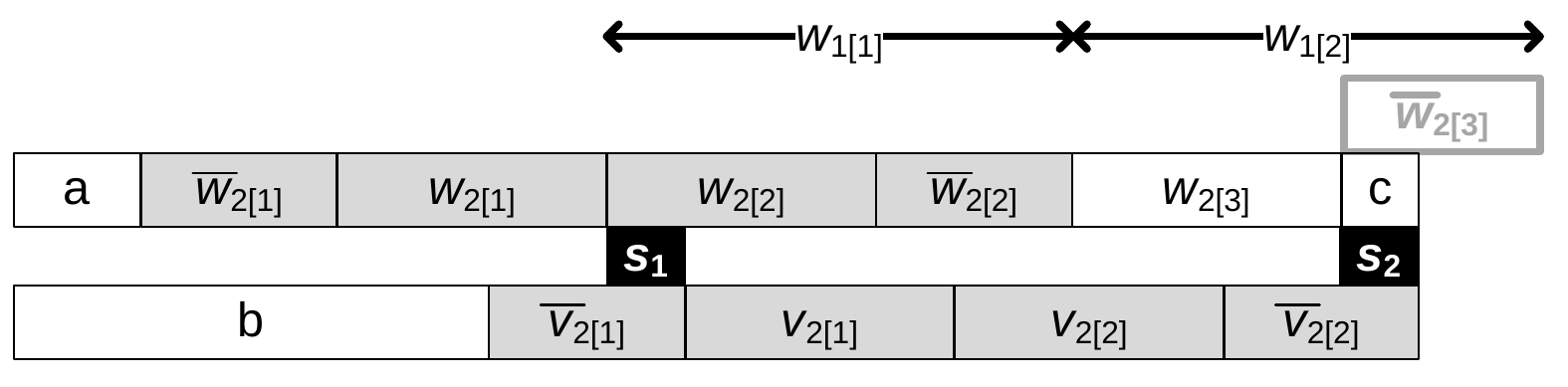}

\noindent  Then $s_1$ is both a prefix of $w_2$ and $\o{w}_2$, 
 contradicting the fact that $lcp(w_2\o{w}_2,\o{w}_2w_2)=0$.

\item Case when $\inv{v}$ starts in the $w_2$ of $w_1$:

\noindent \textit{Note that this covers also the case when 
$\inv{v}$ starts in the second copy of $w_2$ in $\inv{w}$.}

\includegraphics[scale=0.7]{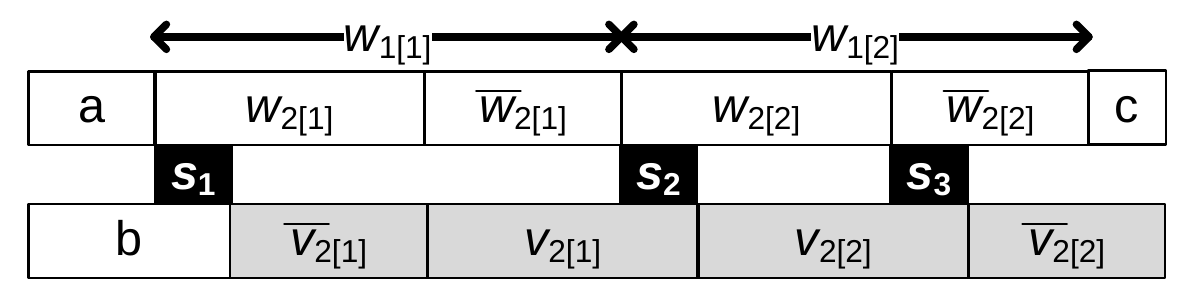}

\noindent  Then $s_1$ is both a prefix of $w_2$ and $\o{w}_2$, 
 contradicting the fact that $lcp(w_2\o{w}_2,\o{w}_2w_2)=0$.\\

\item Case when $\inv{v}$ starts in the $\o{w}_2$ of $w_1$:

\textit{Note that this covers also the case when 
$\inv{v}$ starts in the second copy of $\o{w}_2$ in $\inv{w}$.}

\includegraphics[scale=0.7]{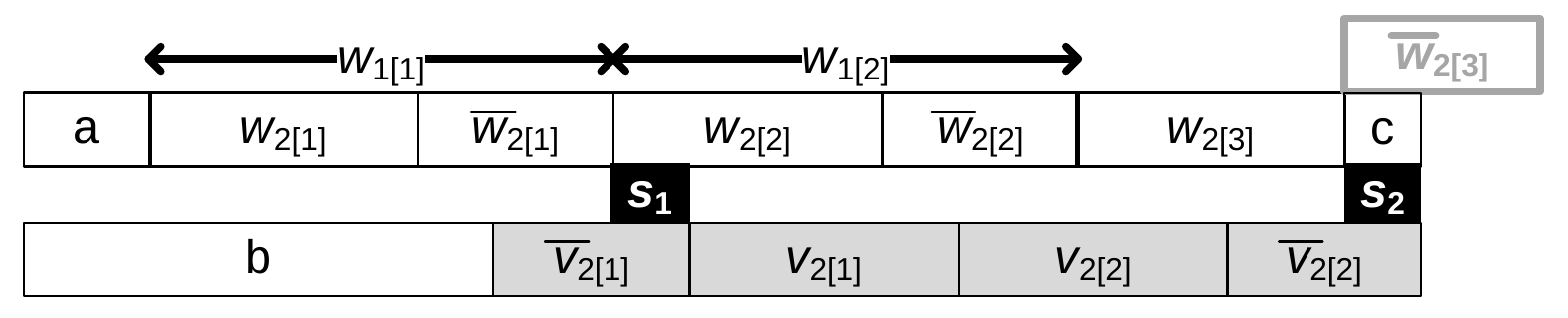}

\noindent  Then $s_1$ is both a prefix of $w_2$ and $\o{w}_2$, 
 contradicting the fact that $lcp(w_2\o{w}_2,\o{w}_2w_2)=0$.\\
\textit{Note that the whole of $\o{w}_{2[3]}$ might not be a part of the string (and that is why in the diagram it is depicted in gray), but then $t$ is a non-trivial proper prefix of $\o{w}_2$ and the 
the argument holds.}

\item case when $\inv{v}$ does not at all overlap with $\inv{w}$.

\noindent
 That case is argued identically as for an \invf\ occurring to the left
of $L_2$.
\end{my_enumerate}
The third step of the proof is to assume by contradiction
that an \invf\ occurs to the left of $L_2$ which follows the same
line of argumentation as the first step.
The fourth and last step of the proof is to assume that an \invf\ occurs
to the right of $R_2$ which follows the same line of argumentation
as for the second step. \qed

\subsection{Proof of Lemma~\ref{v-cases}:}
\label{proof-v-cases}

\smallskip
\leftskip=0pt
\noindent(a) Case $\sbig{v_{[1]}}<R_1$.

\smallskip
\leftskip=10pt
\noindent
Without loss of generality we can assume that \lcp{u} = 0
and hence $R_1=N_1$. If it is not, instead of doing the argument with\newline ${u_1}^\p{U}u_2{u_1}^{(\p{U}\+\q{U})}u_2{u_1}^\q{U}$ we can do the
argument with\newline
$s_2{w_1}^\p{U}w_2{w_1}^{(\p{U}\+\q{U})}w_2{w_1}^{(\q{U}\-1)}s_1$ 
where $w_1$ respective $w_2$ is a right cyclic shift of $w_1$ respective 
$w_2$ by \lcp{u} positions, $s_1s_2=w_1$, and
$|s_1|=$ \lcp{u}. Then $lcp(w_1,\c{w}_1)=0$.
The proof is carried out by a discussion of all possible 
cases of the ending point of $v_{[1]}$. 

\smallskip
\noindent(A) Case $\send{v_{[1]}}\leq \send{u_{[1]}}$

\smallskip
\leftskip=20pt
\noindent 
Note that $\send{v^2}>\send{U_{[1]}}=
\send{{u_1}^\p{U}u_2{u_1}^\q{U}}$, for otherwise there would be a farther copy of $v^2$ in $U_{[2]}$.
By the inversion factor Lemma~\ref{invfactor}, 
 $v_{[1]}$ does not contain the whole of any \invf{s}. Thus,  $v_{[2]}$ cannot contain either the whole of  any \invf{s}, and in particular 
cannot contain the \invf\ at $N_1$. Therefore, $v_{[1]}$ must end in the suffix $\o{u}_2u_2$ of $u_{[1]}$.
Let $s$ be the offset of $v_{[1]}$ in $u_{[1]}$ and let $s_1$ be the overlap between $u_{[1]}$ and $v_{[2]}$, i.e. $svs_1=u={u_1}^\p{U}u_2$, see the diagram bellow for an illustration.

\includegraphics[scale=.8]{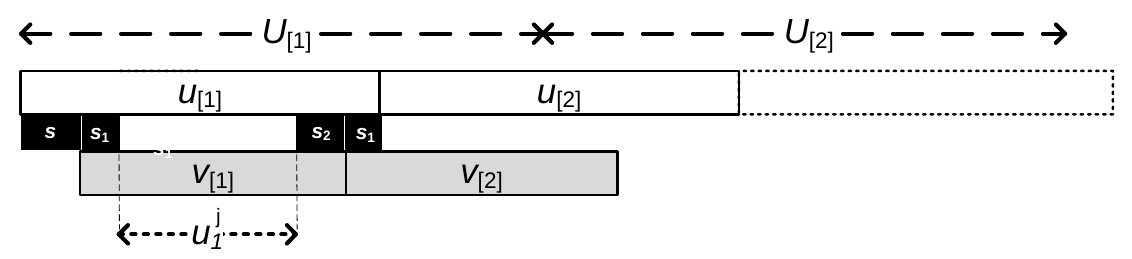}

\noindent
Then  $s_1$ is both a prefix of $v$ and a suffix of $u$. Since $s_1$ is the overlap of $u_1$ and $v_1$, $|s_1|<|u_1|$ and $s_1$ is a suffix of $\o{u}_2u_2$.  It follows that $v=t_1{u_1}^it_2$ for some 
suffix $t_1$ of $u_1$, some prefix $t_2$ of $u_1$, and some $i\geq 0$.\\
On the other hand,  $U_{[1]}={u_1}^\p{U}u_2{u_1}^\q{U}=
u{u_1}^\q{U}$ is
a non-trivial proper prefix of $sv^2$, and so $svs_1{u_1}^\q{U}$ is a non-trivial proper
prefix of $sv^2$, implying that $s_1{u_1}^\q{U}$ is a non-trivial
proper prefix
of $v$ and, therefore, $v=s_1u_i^js_2$ for some prefix $s_2$ of
$u_1$ and some $j\geq 1$.\\
Thus, $v=t_1{u_1}^it_2=s_1{u_1}^js_2$. Since $t_1$ is a suffix of $u_1$
and $t_2$ a prefix of $u_1$, by Lemma~\ref{syncpr},
$t_1=s_1$ and $t_2=s_2$. Therefore, $s_1$ is a suffix of $u_1$. \\
Since $s_2s_1$ is a suffix of $u$, then $s_2s_1={u_1}^iu_2$ for some $i\geq 0$. Since $|s_2|\+|s_1|<2|u_1|$, either $i=0$ or $i=1$, which
proves that either $s_2s_1=u_1u_2$ or $s_2s_1=u_2$.\\
In the former case, $|v|=(j\+1)|u_1|\+|u_2|$ and so
$v={{\widehat u}_1}^{(j\+1)}{\widehat u}_2$,
while in the latter case $v={{\widehat u}_1}^{\ j}{\widehat u}_2$, where 
in both cases
${\widehat u}_1$ respective
${\widehat u}_2$ is a left cyclic shift of $u_1$ respective $u_2$
by $|s_1|$ positions.
The left cyclic shift is possible as $s_1$ is both a suffix of $u_1$ and a suffix of $\c{u}_1=\o{u}_2u_2$. Therefore,
$v={{\widehat u}_1}\np{j}{{\widehat u}_2}$ and $1\leq j\leq \p{U}$ and so when $j<\p{U}$, case $(a_1)$ holds true, and when $j=\p{U}$, case $(a_2)$ holds true.

\leftskip=10pt
\smallskip
\noindent (B) Case $\send{u_{[1]}} < \send{v_{[1]}}\leq \send{u_{[1]}u_1}$

\leftskip=20pt
\smallskip
\noindent 
We discuss this case in four different configurations based on
where $v_{[1]}$ starts and where it ends.
\begin{my_itemize}
\leftskip=10pt
\item[$(1)$] A configuration when $v_{[1]}$ starts in a $u_2$ and ends in the first $u_2$ of $u_{[2]}$.\\
Let $s_1$ be the offset of $v_{[1]}$ in the $u_2$ it starts in, let $s_2$ be the
overlap of $v_{[1]}$ and  the $u_2$ it starts in, let $t_1$ be the overlap
of $v_{[1]}$ with the $u_2$ it ends in, and let $t_2$ be the overlap
of $v_{[2]}$ with the $u_2$ where $v_{[1]}$ ends. Let
${{\widehat u}_1}=s_2\o{u}_2s_1$; as a conjugate of $u_1$,
it is primitive.

\includegraphics[scale=0.75]{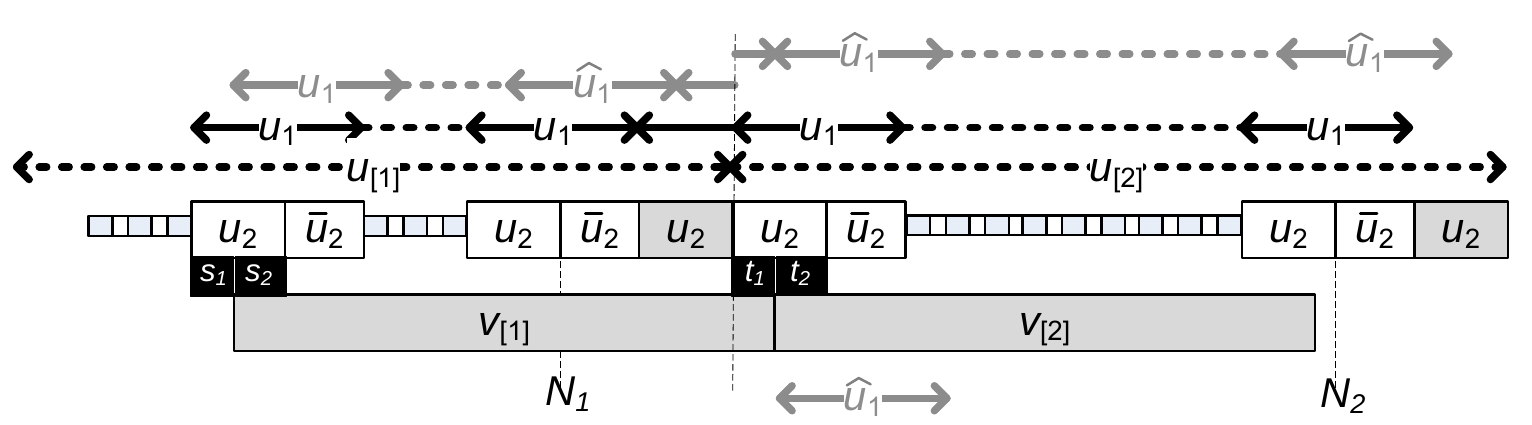}

\noindent
By Lemma~\ref{syncpr}, $t_1=s_1$ and $t_2=s_2$, and
so $s_1v_{[2]}$ is a non-trivial proper prefix of ${u_1}^{(\p{U}\+\q{U})}u_2$.
It follows that the suffix $u_2s_1$ of $v$ must align with
$u_1u_2=(u_2\o{u}_2)u_2$ of ${u_1}^{(\p{U}\+\q{U})}u_2$, and so 
$s_1$ is prefix of $u_2\o{u}_2$. Thus, $s_1$ is a prefix of both, $u_2\o{u}_2$ and $\o{u}_2u_2$. Therefore,
$|s_1|\leq$ \lcp{u} = 0, and so $s_1$ is empty.
It follows that $v=u_1^ju_2$ for $1\leq j\leq \p{U}$ and so either
$(a_1)$ or $(a_2)$ holds true.

\textit{Note that Lemma~\ref{syncpr} applies even if $\p{U}=1$, since
then $v_{[1]}$ must start in the very first $u_2$ of $u_{[1]}$}.

\item[$(2)$] A configuration when $v_{[1]}$ starts in a $u_2$ and ends in the first $\o{u}_2$ of $u_{[2]}$.\\
Let $s_1$ and $s_2$ be as in the previous case (B)(1).  Let $t_1$
be the overlap of $v_{[1]}$ and $u_{[2]}$.

\includegraphics[scale=0.75]{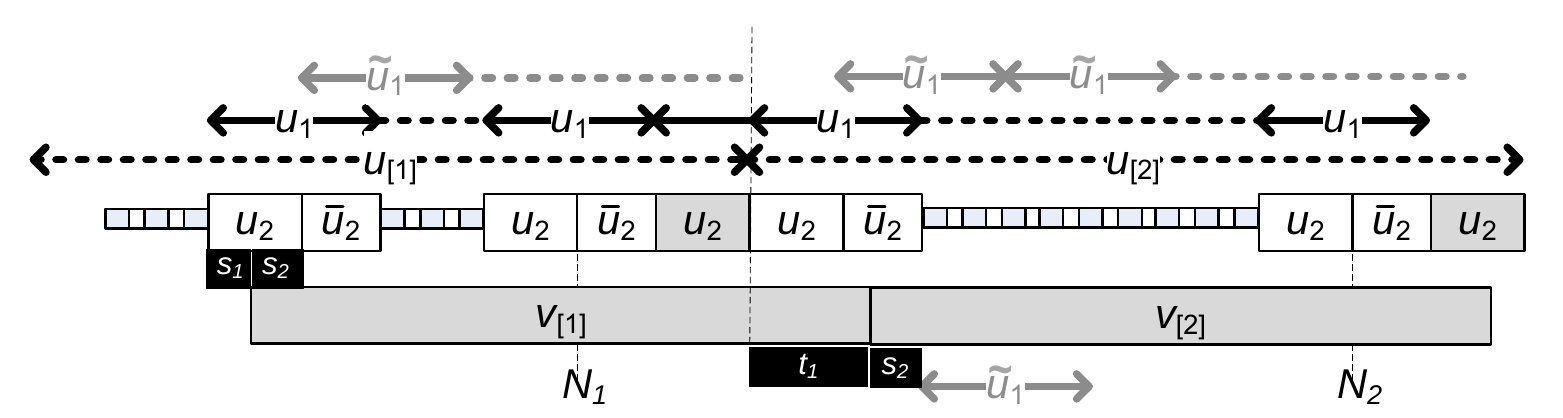}

\noindent The factor ${u_1}^{(\p{U}|+\q{U})}u_2$ has $u_2\c{u}_1\c{u}_1$
as a prefix as $\p{U}\+\q{U}\geq 2$.
The factor $v$ has $s_2\c{u}_1$ as
a prefix. Thus ${u_1}^{(\p{U}|+\q{U})}u_2$ has also 
$t_1s_2\c{u}_1$ as a
prefix. Since $|t_1s_2|<|u_2|\+|u_1|$, this contradicts Lemma~\ref{syncpr}, as $\c{u}_1$ is primitive being a conjugate of $u_1$.
Such a configuration is not possible.

\item[$(3)$]  A configuration when $v_{[1]}$ starts in a $\o{u}_2$ and ends in the first $u_2$ of $u_{[2]}$.\\
Let $s_1$ be the offset of $v_{[1]}$ in $\o{u}_2$ it starts in, let 
$s_2$ be the overlap of $v_{[1]}$ and  the $\o{u}_2$ it starts in.
Let $t_1$ be the overlap of $v_{[1]}$ with $u_{[2]}$

\includegraphics[scale=0.75]{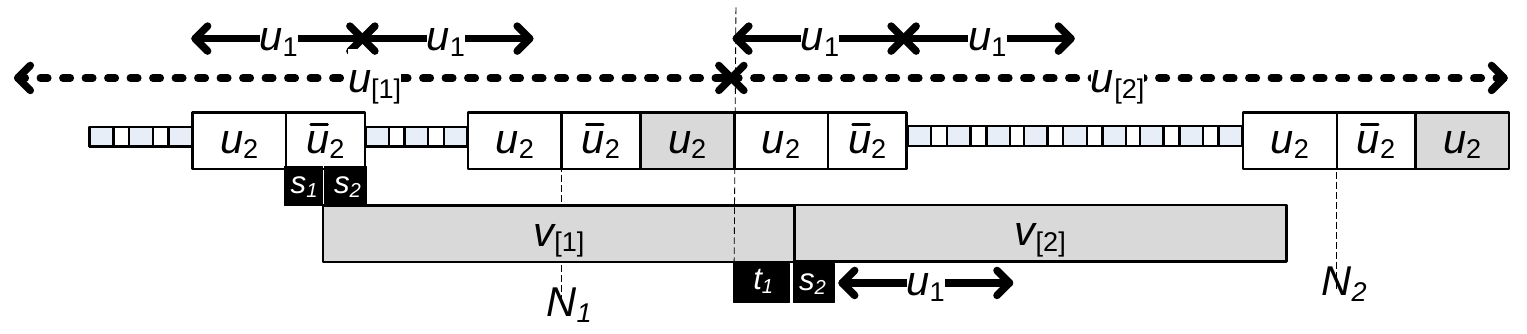}

\noindent The factor $v$ has $s_2u_1$ as a prefix, and
so ${u_1}^{(\p{U}\+\q{U})}$ has as a prefix
$u_1u_1$ and $t_1s_2u_1$. Since $|t_1s_2|<|u_1|$, this
contradicts Lemma~\ref{syncpr}. Such a configuration is not  possible.

\item[$(4)$] A configuration when $v_{[1]}$ starts in a $\o{u}_2$ and ends in the first $\o{u}_2$ of $u_{[2]}$.\\
Let $s_1$ and $s_2$ be as in (B)(3). Let $t_1$ be the overlap of $v_{[1]}$
and the $\o{u}_2$ it ends in, and let $t_2$ be the overlap of
$v_{[2]}$ with the $\o{u}_2$ in which $v_{[1]}$ ends.

\includegraphics[scale=0.75]{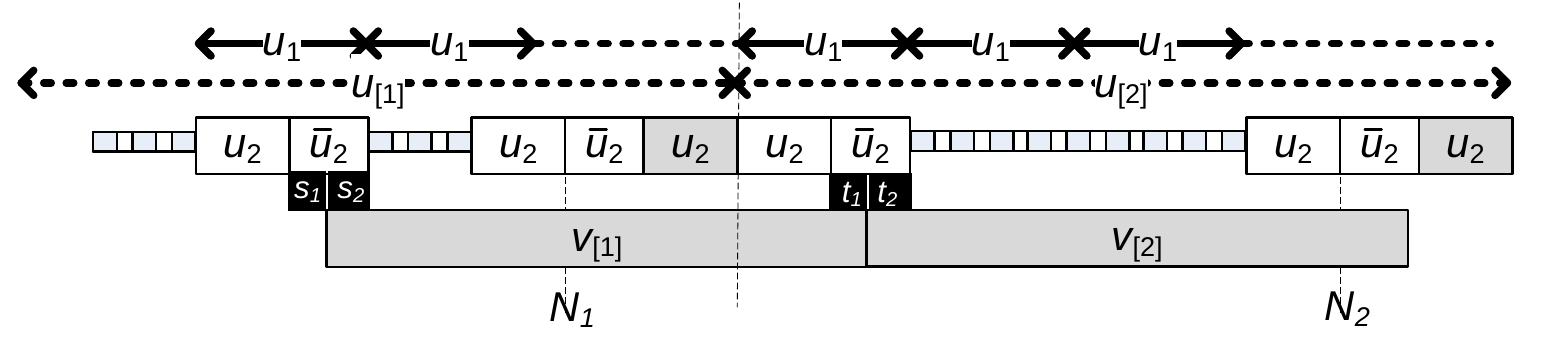}

\noindent  By Lemma~\ref{syncpr}, $t_1=s_1$ and $t_2=s_2$. Since $u_2s_1v_{[2]}$
is a prefix of ${u_1}^{(\p{U}\+\q{U})}u_2$, it follows that the suffix
$u_2u_2s_1$ of $v_{[2]}$ must align with $u_1u_2$ in ${u_1}^{(\p{U}\+\q{U})}u_2$, 
and thus
$u_2u_2s_1$ is a prefix of $u_2\o{u}_2u_2$, hence  $u_2s_1$ is a prefix of $\o{u}_2u_2$. Thus,
$u_2\o{u}_2=u_2s_1s_2$ is a prefix of $\o{u}_2u_2s_2$, giving
$u_2\o{u}_2=\o{u}_2u_2$, which is a contradiction as
$u_2\o{u}_2$ is primitive. Such a configuration is not  possible.
\end{my_itemize}

\leftskip=10pt
\noindent (C) Case $\send{u_{[1]}u_1} < \send{v_{[1]}}<R_2$.

\leftskip=20pt
\smallskip
\noindent Then $v_{[1]}$ contains the \invf\ at $R_1$. Thus, 
$v_{[2]}$ must contain the \invf\ at $R_2$ and it must be placed
in $v_{[2]}$ in the same position mate to the beginning of
$v_{[2]}$ as in $v_{[1]}$, and therefore 
$|v|=R_2\-R_1=|U|$. Thus, case $(a_4)$ holds
true.

\leftskip=10pt
\smallskip
\noindent (D) Case $R_2\leq \send{v_{[1]}}$.

\leftskip=20pt
\smallskip
\noindent
Since $\send{v_{[1]}}\geq R_2\geq N_2={u_1}^{\p{U}}u_2{u_1}^{(\p{U}\+\q{U}\-1)}{u_2}$,
either\newline
$s_1\o{u}_2u_2{u_1}^{(\p{U}\+\q{U}\-1)}{u_2}$
for some suffix $s_1$ of $u_2$ is a prefix of $v$,
or\newline
$s_1{u_1}^iu_2{u_1}^{(\p{U}\+\q{U}\-1)}{u_2}$
for some suffix $s_1$ of $u_1$ and some $i\geq 1$ is 
a prefix of $v$, 
 and so case $(a_5)$ holds true. 

\leftskip=10pt
\smallskip
\noindent
Case $(a_3)$ is not possible as it never materialized during the discussion of the cases\linebreak
(A) - (D) that cover exhaustively all possible endings of $v_{[1]}$.

\leftskip=0pt
\smallskip
\noindent(b) Case $\send{v_{[1]}}\leq \send{u_{[1]}}$.

\smallskip
\leftskip=10pt
\noindent
If $\sbig{v_{[1]}}\geq R_1$, then $|v|<|u_1|$ and so $v^2$ is a factor of $u_1u_2{u_1}$ and hence of $U_{[1]}$, and thus there is a farther
copy of $v^2$ in $U_{[2]}$, a contradiction. Therefore
$\sbig{v_{[1]}}<R_1$ and this is the case (A) above, and thus
the either the case $(a_1)$ or case $(a_2)$ holds.
\qed

\leftskip=0pt
\subsection{Proof of Lemma~\ref{vv-cases}:}
\label{proof-vV-cases}

\smallskip
\leftskip=0pt
\noindent Case $(a)$:
since $\sbig{{v^2}}=\sbig{{V^2}}=\sbig{{\cal V}}\leq R_1(\cal U)$,
applying Lemma~\ref{v-cases} to $v^2$ and $V^2$ gives the following possibilities:

\vspace{-5pt}
\begin{my_itemize}
\leftskip=3pt
\item[$(i)$] $v={{\widehat u}_1}\np{i}{\widehat u}_2$ for $1\leq i<\p{U}$ 
where ${\widehat u}_2$ is a non-trivial proper prefix of ${\widehat u}_1$ and where ${\widehat u}_1$ respective 
${\widehat u}_2$ is a cyclic shift of $u_1$ respective
$u_2$ in the same direction by the same number of positions
(by item $(a_1)$ of Lemma~\ref{v-cases} applied to $v^2$),
\item[$(ii)$]
$v={{\widehat u}_1}^{\p{U}}{\widehat u}_2$  
where ${\widehat u}_2$ is a non-trivial proper prefix of ${\widehat u}_1$ and where ${\widehat u}_1$ respective 
${\widehat u}_2$ is a cyclic shift of $u_1$ respective
$u_2$ in the same direction by the same number of positions
(by item $(a_2)$ of Lemma~\ref{v-cases} applied to $v^2$),
\item[$(iii)$] $|v|=|U|$  
(by item $(a_4)$ of Lemma~\ref{v-cases} applied to $v^2$),
\item[$(iv)$] $\send{v_{[1]}}\-\send{u_{[1]}}\geq(\p{U}\+\q{U}\-1)|u_1|\+|u_2|$
(by item $(a_5)$ of Lemma~\ref{v-cases} applied to $v^2$),
\item[$(I)$] $V={{\widehat u}_1}\np{j}{\widehat u}_2$ for $1\leq j<\p{U}$ 
where ${\widehat u}_2$ is a non-trivial proper prefix of ${\widehat u}_1$ and where ${\widehat u}_1$ respective 
${\widehat u}_2$ is a cyclic shift of $u_1$ respective
$u_2$ in the same direction by the same number of positions
(either by item $(a_1)$ or $(a_2)$ of Lemma~\ref{v-cases} applied to $V^2$),
\item[$(II)$] $|V|=|U|$  
(by item $(a_4)$ of Lemma~\ref{v-cases} applied to $V^2$),
\item[$(III)$] Either $s_1\o{u}_2u_2{u_1}^{(\p{U}\+\q{U}\-1)}{u_2}$
for some suffix $s_1$ of $u_2$ is a prefix of $V$,
or\linebreak
 $s_1{u_1}^iu_2{u_1}^{(\p{U}\+\q{U}\-1)}{u_2}$
for some suffix $s_1$ of $u_1$ and some $j\geq 1$ is 
a prefix of $V$
(by item $(a_5)$ of Lemma~\ref{v-cases} applied to $V^2$).
\end{my_itemize}

\noindent 
We inspect all possible combinations:
\begin{my_itemize}
\leftskip=-10pt
\item[$\cdot$] 
Combining $(i)$ and $(I)$ is impossible: since $v$ is a prefix of
$V$, ${{\widehat u}_1}={{\widehat u}_1}$
and ${{\widehat u}_2}={{\widehat u}_2}$.  Since $j>i$ as $|V|>|v|$, 
we can apply Lemma~\ref{canon1} deriving a contradiction.
\item[$\cdot$]
Combining $(i)$ and $(II)$ is possible and yields case $(a_2)$:  since $v$ is a prefix of $V$, ${{\widehat u}_1}={{\widehat u}_1}$
and ${{\widehat u}_2}={{\widehat u}_2}$  and so $\cal V$ must be a \bmate\ of $\cal U$. 
Since $|V|=|U|=(\p{U}\+\q{U})|u_1|\+|u_2|$, 
$V={\widehat u}_1^{\ i}{\widehat u}_2{\widehat u}_1^{(\q{U}\+\p{U}\-i)}$. Since $i\geq \q{U}\+\p{U}\-i$ as otherwise there would
be a farther copy of $v^2$,  $2i\geq\p{U}\+\q{U}$. Since
$1\leq i<\p{U}$, $i = \p{U}\-k$ for some $1\leq k < \p{U}$.
It follows that 
$2(\p{U}\-k)\geq\p{U}\+\q{U}$, so
$2\p{U}\-2k\geq  = \p{U}\+\q{U}$, and thus $\p{U}\geq \q{U}\+2$.
\item[$\cdot$]
Combining $(i)$ and $(III)$ is impossible: since $v^2$ is a prefix of $V^2$,
${{\widehat u}_1}\np{i}{{\widehat u}_2}{{\widehat u}_1}\np{i}{{\widehat u}_2}$ is a prefix of $V^2$. At the same time either $s_1{u_1}^j{u_2}{u_1}^{(\p{U}\+\q{U}\-1)}u_2$
is a prefix of $V$ or $s_1\o{u_2}{u_2}{u_1}^{(\p{U}\+\q{U}\-1)}u_2$ is a  prefix of $V$. Due to Lemma~\ref{syncpr}, in both cases, ${{\widehat u}_1}\np{i}{{\widehat u}_2}{{\widehat u}_1}^{(\p{U}\+\q{U}\-1)}{{\widehat u}_2}$ is a prefix of $V$
and so ${{\widehat u}_1}\np{i}{{\widehat u}_2}{{\widehat u}_1}\np{i}{{\widehat u}_2}$ is a prefix of $V$. It follows that $v^2$ is a factor in $V_{[1]}$ and, consequently,  it has a farther copy in  $V_{[2]}$, a contradiction.
\item[$\cdot$]
Combining $(ii)$ and $(I)$ is impossible: as $j\leq \p{U}$ implies that $|V|\leq |v|$, hence a contradiction.
\item[$\cdot$]
Combining $(ii)$ and $(II)$ is possible and yields that  $\cal V$ is an \amate\ 
of $\cal U$, hence case $(a_1)$.
\item[$\cdot$]
Combining $(ii)$ and $(III)$ is impossible for the same reasons
as for the combination $(i)$ and $(III)$.
\item[$\cdot$]
Combining $(iii)$ and $(I)$ or $(II)$ is impossible due to the size of $v$ being bigger
than the size of $V$.
\item[$\cdot$] 
Combining $(iii)$ and $(III)$ is possible and yields case $(a_3)$ and
so $\cal V$ is a \gmate\ of $\cal U$.
\item[$\cdot$] 
Combining $(iv)$ and $(I)$ or $(II)$  is impossible due to the size of $v$ being bigger than the size of $V$.
\item[$\cdot$]
Combining $(iv)$ and $(III)$ yields case $(a_4)$.
\end{my_itemize}

\noindent Case $(b)$: The \fsds\  $\cal V$ is an \emate\ of $\cal U$
by definition as $R_1\leq \sbig{{\cal V}}$. If $\send{v_{[1]}}\leq \send{u_{[1]}}$, then by Lemma~\ref{v-cases}, $\sbig{{\cal V}}<R_1$, a contradiction.
So $\send{u_{[1]}}<\send{v_{[1]}}$.
\qed

\bigskip
\bigskip
\noindent
{\bf Acknowledgments.} 
The authors would like to thank  Nguyen Huong Lam and the anonymous referees for valuable comments and suggestions which improved the quality of the paper. 
This work was supported by grants from the Natural Sciences and Engineering Research Council of Canada, MITACS, and by
the Canada Research Chairs program.

\bigskip
\noindent {\bf References}
\bibliographystyle{plain}
\bibliography{da3474}    

\end{document}